\documentclass{article}
\usepackage{graphicx} 

\usepackage[utf8]{inputenc}
\usepackage{amsfonts}
\usepackage{amsmath}
\usepackage{amsthm}
\usepackage{amssymb}
\usepackage{latexsym}
\usepackage{enumerate}
\usepackage{centernot}
\usepackage{amscd}
\usepackage{hyperref}
\usepackage{extpfeil}

\newtheorem{theorem}{Theorem}[section]
\newtheorem{lemma}[theorem]{Lemma}
\newtheorem{proposition}[theorem]{Proposition}
\newtheorem{corollary}[theorem]{Corollary}

\newtheorem{remark}[theorem]{Remark}

\newtheorem{question}[theorem]{Question}

\newtheorem{definition}[theorem]{Definition}
\newtheorem{obs}[theorem]{Remark}

\newtheorem{defi}[theorem]{Definition}

\newtheorem{pro}[theorem]{Proposition}

\newcommand{\pin}[2]{\langle #1, #2 \rangle}

\newcommand{\Sp}{\mathrm{Sp}}

\newcommand{\N}{\mathbb{N}}
\newcommand{\R}{\mathbb{R}}

\newcommand{\C}{\mathbb{C}}
\newcommand{\Z}{\mathbb{Z}}
\newcommand{\U}{\mathcal U}

\newcommand{\Id}{\operatorname{Id}}

\title{On complex structures and  uniqueness of
algebra norms in Banach spaces.}

\author{W. Cuellar Carrera  \and  V. Ferenczi}

\date{}

\begin{document}

\maketitle

\begin{abstract} For $X$ an infinite dimensional Banach space, we contribute to the study of the Banach algebra 
$L(X)/S(X)$, where $S(X)$ is the ideal of strictly singular operators.
    We extend results of Ferenczi-Galego \cite{FGEven} by proving that $\|I-J\|_S \geq 2$, whenever $I$ is a complex structure on a real space $X$ and $J$ extends a complex structure on a hyperplane of $X$, and where $\|.\|_S$ denotes a certain algebra norm on $L(X)/S(X)$ dominated by the usual quotient norm $\|.\|$. We solve two questions of Kalton-Swanson \cite{kaltonasymplectic} by proving that if $X=Z_2$ the Kalton-Peck space, then $L(Z_2)/S(Z_2)$ a) is not complete for $\|.\|_S$ and
    b) that it is not *-isomorphic to a $C^*$-algebra for $\|.\|$. In particular $L(Z_2)/S(Z_2)$ admits two inequivalent *-algebra norms.

\end{abstract}

\let\thefootnote\relax 

   \let\thefootnote\relax
    
\let\thefootnote\relax\footnotetext{Both authors were supported by the São Paulo Research Foundation (FAPESP), grant no. 2023/12916-1. V. Ferenczi was also supported by CNPq, grant no. 304194/2023-9.}

\section{Introduction}

\subsection{Operator Ideals}
In this work, Banach spaces are assumed to be either real or complex. However, when dealing with complex structures, the spaces will usually assumed to be real. The term subspace will always refer to a closed linear subspace, unless explicitly stated otherwise.
 Operator ideals, in the sense of Pietsch \cite{Pietsch1980}, are classes of  operators between Banach spaces that are closed under addition, composition with bounded operators,  and that contain all finite-rank operators. A proper operator ideal contains only finite-rank invertible operators. Classical examples include the ideals of compact operators ($\mathcal{K}$), strictly singular operators ($\mathcal{S}$), and inessential operators ($\mathcal{IN}$).

 Let $L(X,Y)$ denote the space of bounded linear operators between Banach spaces $X$ and $Y$.  An operator $T \in L(X,Y)$ is \textit{strictly singular} if its restriction to any infinite-dimensional subspace of $X$ fails to be an isomorphism onto its range. The collection of all such operators will be denoted by $S(X,Y)$; when $Y = X$, we simply write $S(X)$.
$T\in L(X,Y)$ is said \textit{finitely singular} (or upper semi-Fredholm) if  there exists a finite codimensional subspace $Z$ of $X$ such that $T|_Z$ is an isomorphism onto its range $TZ$. It is  \textit{infinitely singular} otherwise, which is equivalent to saying that for every $\varepsilon >0$ there exists an infinite dimensional subspace $Z \subset X$ such that $\|T|_Z\| < \varepsilon$ \cite[Proposition 3.2]{Maurey}. 
When $T$ is finitely singular then its index is defined as $\operatorname{ind}(T) := \dim(\ker T) - \dim(\mathrm{coker}\, T)$, and is well-defined in $\Z \cup \{-\infty\}$; furthermore this index is continuous (\cite[Corollary 4.1]{Maurey}).

The ideal $\mathcal{IN}$ of \textit{inessential operators} consists of operators $R \in L(X,Y)$ for which $\Id_X - TR$ is Fredholm for all $T \in L(Y,X)$ (or equivalently, $\Id_Y - RT$ is Fredholm for all $T \in L(Y, X)$).  Here, an operator is \textit{Fredholm} if it is upper semi-Fredholm and its image has finite codimension. The index of a Fredholm operator $T$ is therefore an element of $\Z$.  A key property of $\mathcal{IN}$ is the stability of Fredholm operators under perturbations: if $F \in L(X,Y)$ is Fredholm and $R \in \mathcal{IN}(X,Y)$, then $F + R$ remains Fredholm (with the same index). These ideals satisfy the inclusion $\mathcal{K} \subseteq \mathcal{S} \subseteq \mathcal{IN}$.

A fundamental question posed by Pietsch \cite{Pietsch1980} asked whether $\mathcal{IN}$ is the largest proper operator ideal. The second name author \cite[Theorems 2--3]{Ferenczi2023} solved this problem negatively, demonstrating that no largest proper operator ideal (in the sense of Pietsch) exists.

\subsection{Twisted sums}
We refer the reader to \cite{CabelloCastillo} for a comprehensive treatment of homological theory in Banach spaces. A short exact sequence of quasi-Banach spaces is a diagram of quasi-Banach spaces and operators 

 \begin{equation} \label{exact1} \begin{CD} 
0@>>>  Y@>i>> X @>\pi>> Z@>>> 0\end{CD} \end{equation}
such that the image of each arrow coincides with the kernel of the subsequent one. The middle space $X$ is usually called a \textit{twisted sum} of $Y$ and $Z$. It follows from the open mapping theorem that $X$ contains a copy of $Y$ such that the corresponding quotient is isomorphic to $Z$.
Two short exact sequences are considered equivalent provided that the subspace and quotient are fixed:
\begin{definition}
Two exact sequences 
$$\begin{CD}
0@>>>  Y@>i_1>> X_1 @>\pi_1>> Z@>>> 0,\end{CD}$$

$$\begin{CD}
0@>>>  Y@>i_2>> X_2 @>\pi_2>> Z@>>> 0,\end{CD}$$
are \textit{equivalent} if there exists an isomorphism $T: X_1 \to X_2$ such that the diagram is commutative
$$\begin{CD}
0@>>>  Y@>i_1>> X_1 @>\pi_1>> Z@>>> 0\\
 &&@| @VVTV @| \\
0@>>> Y@>i_2>>  X_2@>\pi_2>> Z@>>>0
\end{CD}$$
\end{definition}
The exact sequence  (\ref{exact1}) is \textit{trivial} if it is equivalent to $$\begin{CD}
0@>>>  Y@>>> Y\oplus Z @>>> Z@>>> 0\end{CD},$$ 
which occurs exactly when $i(Y)$ is complemented in $X$ or there exists a bounded linear section $s:Z\to X$ ($\pi\circ s= \Id_Z$) \cite[Lemma 2.1.5]{CabelloCastillo}. 

According to the theory developed by Kalton and Peck \cite{KaltonPeck}, twisted sums of quasi-Banach spaces are in correspondence with certain nonlinear maps.

\begin{definition}
    A homogeneous map $\Omega:Z\to Y$ is \textit{quasi-linear} if 
    \[
\|\Omega(z_1+z_2) - \Omega(z_1) - \Omega(z_2)\| \leq C(\|z_1\| + \|z_2\|)
\]
for some constant $C>0$ and every $z_1,z_2 \in Z$. Denote by $c(\Omega)$ the optimal constant $C$ that satisfies the condition above. 
\end{definition}  

Let us denote  by $Y \oplus_\Omega Z$  the product space $Y\times Z$  endowed with the quasi-norm $\|(y,z)\|_\Omega = \|y - \Omega z\|_Y + \|z\|_Z$. In this case, the sequence
$$\begin{CD}
0@>>>  Y@>i>> Y\oplus_\Omega Z @>\pi>> Z@>>> 0\end{CD},$$ 

where $i(y)=(y,0)$ and $\pi(y,z)=z$, is exact making 
$Y$  isometric to its copy in $Y \oplus_\Omega Z$  such that the respective quotient is also isometric to $Z$.
Kalton and Peck proved that every short exact sequence is equivalent to one of the type above for some quasi-linear map. 

A quasi-linear map $\Omega$ is trivial when the correspondent twisted sum is trivial, which is equivalent to write $\Omega = L + B$ for some $L$  linear and $B$  bounded maps \cite[Lemma 3.3.2]{CabelloCastillo}.

In \cite{KaltonPeck}, Kalton and Peck introduced a construction of nontrivial quasi-linear maps $\ell_p \to \ell_p$ for the entire range $0 < p < \infty$. The fundamental examples arising from this theory are the \emph{Kalton--Peck spaces} $Z_p = \ell_p \oplus_{\Omega_p} \ell_p$, where the quasi-linear map $\Omega_p$ is defined on vectors of finite support by
\[
\Omega_p(x)(n) = x(n)\log\frac{\|x\|_p}{|x(n)|}.
\]
This construction yields a nontrivial twisted sum of $\ell_p$ with itself (see also \cite[Proposition~3.2.7]{CabelloCastillo}  for more details). Observe that $Z_p$ is a quasi-Banach space for $0 < p \le 1$, while, in general, any twisted sum of $\ell_p$ with itself can be renormed to be a Banach space for $1 < p < \infty$ (\cite[Proposition~3.4.5]{CabelloCastillo}).

An important class of nontrivial exact sequences are those for which the quotient map is a strictly singular operator. 

\begin{definition}
    An exact sequence $ 0 \to   Y \to  X \to  Z \to 0$ is \textit{singular} if the quotient map is a strictly singular operator. A quasi-linear map  $\Omega:Z\to Y$ is singular if the corresponding exact exact sequence  $ 0 \to   Y \to  Y\oplus_\Omega Z \to  Z \to 0$ is singular.
\end{definition}
It follows that a quasi-linear map $\Omega:Z\to Y$ is singular if and only if its restriction $\Omega|_W$ remains non-trivial for every infinite dimensional subspace $W \subseteq Z$ (i.e, the copy of $Y$ is not complemented $\pi^{-1}(W)$) \cite[Proposition 9.1.2] {CabelloCastillo}. It was shown by Kalton and Peck \cite[Theorem 6.4]{KaltonPeck} that the quasi-linear maps $\Omega_p$ are singular for $1<p<\infty$. The same holds for $p=1$ \cite{CMp1} and $0<p<1$ \cite{CCS}.

 Kalton and Peck also established a duality for the Banach spaces $Z_p$ ($1<p<\infty$). Specifically, they proved in \cite[Theorem 5.1]{KaltonPeck} (See also \cite[Proposition 3.8.5]{CabelloCastillo}) that the dual space $Z_p^*$ is isomorphic to $Z_q$, where $\tfrac{1}{p} + \tfrac{1}{q} = 1$. The duality is realized through the bilinear form
\[
B((w,z),(x,y)) = \langle y, w \rangle  - \langle x, z \rangle,
\]
for $ (w,z)  \in Z_p$ and $ (x,y)\in Z_q$. 
In this framework, for every bounded linear operator $S \in L(Z_p)$, one can define its adjoint operator $S^* \in L(Z_q)$ through the above duality. More precisely, $S^*$ is the unique operator satisfying
\[
B(S(w,z),(x,y)) = B((w,z),S^*(x,y))
\qquad \text{for all } (w,z)  \in Z_p,\;  (x,y) \in Z_q.
\]

An important open question is whether $Z_2$ is isomorphic to its hyperplanes. Although Banach's general hyperplane problem was solved by Gowers \cite{gowers}, see also the class of HI spaces defined by Gowers and Maurey \cite{gowersmaurey}, the case of $Z_2$ remains famously open.

\begin{question}\label{hyp} Show that Kalton-Peck space $Z_2$ is not isomorphic to its hyperplanes.
\end{question}

\subsection{Complex structures}
A complex structure on a real Banach space $X$ is a structure that endows $X$ with a structure of complex vector space, preserving the original real structure and its topology. Note that complex structures could be considered on complex spaces as well, by only considering the underlying $\R$-linear structure; but in this setting the original complex structure does not play any specific role. So in what follows in this subsection, we assume $X$ is real.

Up to renormings, complex structures on $X$ are in a bijective correspondence with $\R$-linear operators $J: X\to X$ such that $J^2=-\Id$. For each such operator $J$, we denote by $X^J$ the space $X$ with the $\C$-linear structure defined by $ix:= Jx$ endowed with the (equivalent) norm
\[\|| x |\| = \sup_{0\leq  \theta \leq 2\pi} \| \cos \theta x + \sin \theta Jx\|. \]
One canonical construction is the complexification $X_\mathbb{C} = (X \oplus X)^J$ associated to the complex structure  $J(x, y) = (-y, x)$ on $X^2$. 
For an operator $T: X \to Y$ between real Banach spaces,  $\widehat{T}\in L(X_\C, Y_\C)$ denotes its complexification defined by $\widehat{T}(x,y)= (Tx, Ty)$.

Two complex structures $I$ and $J$ on $X$ are equivalent if there exists a $\R$-linear isomorphism $T: X \to X$ such that $J= T^{-1} I T$, meaning the associated complex spaces $X^I$ and $X^J$ are $\C$-linear isomorphic through $T$. This tool is central in the classification of complex structures up to isomorphism. For instance, the real spaces $\ell_p$, $L_p(0,1)$, $c_0$, $C([0,1])$ have unique complex structure (a result of Kalton appearing in \cite{FG}).

A lot of structure exists on the real space $Z_2$. For example as a space isomorphic to its square, $Z_2$ admits  complex structures \cite{CCFM}. It also admits a non-trivial symplectic structure \cite{kaltonasymplectic}. By studying these in more details one may hope to be able to prove that the hyperplanes of $Z_2$ lack at least one of these structures, thus answering Problem \ref{hyp}. For example, Ferenczi-Galego \cite{FGEven} define a space to be even if it admits complex structures but its hyperplanes do not:

\begin{question} Prove that the real  version of Kalton-Peck space is even.
\end{question}

In \cite{CCFM}, as a partial answer, it is proved that no complex structure on the hyperplanes of $Z_2$ containing the copy of $\ell_2$ can leave this copy invariant. An intermediary result is that for a singular twisted sum $X=Y \oplus_\Omega Z$, a complex structure on $Y$ cannot extend simultaneously to complex structures on $X$ and on a hyperplane of $X$.

In this paper we investigate further relations between complex structures on a space and its hyperplane, proving an extension of the above result of \cite{CCFM} to $\varepsilon$-singular twisted sums for an appropriate notion of $\varepsilon$-singularity.
This is based on an extension of a result of \cite{FGEven}, in which we replace an algebraic estimate by a universal quantitative estimate for the difference between a complex structure $I$ on $X$ and a complex structure $J$ on one of its hyperplanes.
Precisely, for an operator $T \in L(Y,X)$, let $\|T\|_S$ denotes the infimum of $\varepsilon>0$ such that any  subspace of $Y$ admits a subspace $Z$ such that $\|T_{|Z}\| \leq \varepsilon$. We prove:

\begin{theorem}\label{16} If $X$ is a real Banach space, $I$ a complex structure on $X$ and $J$ a complex structure on its hyperplane $H$, then
$\|I_{|H}-J\|_S \geq 2$.
\end{theorem}

\subsection{Uniqueness of algebra norms}

For references on recent developments on the problem of uniqueness of algebra norms, for algebras related to algebra of operators on Banach spaces, we used \cite{W,niels}. A Banach algebra $A$ is said to have unique algebra norm if all algebra norms on $A$ are equivalent (to the the original one).
For most classical Banach spaces $X$ (including Kalton-Peck spaces $Z_p$), the algebra $L(X)$
has unique algebra norm; indeed this holds whenever $X$ admits a "continuous bisection of the identity", a condition implied by being isomorphic to its square (\cite{Johnson}, see \cite{W} Corollary 1.1.12).
When $X \not\simeq X^2$, counterexamples exist (the first one seems to be due to Read \cite{Read}, see \cite{W} Theorem 1.1.19). 

Calkin algebras, or algebras of the form $L(X)/I$ for other ideals $I$, are other natural and classical examples. 
Meyer \cite{Meyer} proved the the Calkin algebra $L(X)/K(X)$, where $K(X)$ denotes the compact operators on $X$, has a unique algebra norm when $X=c_0$ or $\ell_p, 1 \leq p <\infty$. The list is extended to other sequence spaces in Ware's thesis \cite{W}, while an example of (exotic) space $X$ for which $L(X)/K(X)$ does not admit a unique norm is due do Tarbard \cite{T}. Recently Johnson and Philips \cite{JP} prove that this holds for
the Calkin algebra of $X=L_p, 1<p<\infty$ as well (the case $p=1$ seems to remain open). 

On the other hand, in the case $p \neq 2$, in \cite{JP} a closed ideal $I(L_p)$ is constructed  so that
$L(L_p)/I(L_p)$ does not admit a unique algebra norm.
We use Kalton-Peck space to provide another example, which additionally is a (noncommutative) *-algebra:

\begin{theorem}\label{17} The algebra $L(Z_2)/S(Z_2)$ admits two inequivalent *-algebra norms.
\end{theorem}

This result is strongly related to an investigation of the non trivial symplectic structure of $Z_2$ (in both the real and complex case) and its relation to a certain representation $\Lambda$, which identifies the quotient space $L(Z_2)/S(Z_2)$ as a *-subalgebra of $L(H)$ defined in \cite{kaltonasymplectic}.
In particular 
we answer a question of Kalton-Swanson \cite{kaltonasymplectic} by proving that this *-subalgebra is not complete. 

\begin{theorem}\label{18} The *-algebra $\Lambda(L(Z_2))$ is not closed in $L(H)$
\end{theorem}

Along the same lines, we answer another question of Kalton-Swanson  \cite{kaltonasymplectic} which also appears as an open question in Castillo-González-Pino \cite{raul}:

\begin{theorem}\label{19} The *-algebra $L(Z_2)/S(Z_2)$, equipped with the quotient norm, is not *-isomorphic to a $C^*$-algebra.
\end{theorem}

\subsection{Organization of the paper}

After this introduction, we give definitions for the two natural quantitative notions of singularity for operators and investigate the relations between them, Subsection 2.1. We also consider a quantitative notion of inessential operators, which is sometimes more precise, Subsection 2.2.
In the third section we present new results about complex structures. A first series of results gives conditions for a perturbation of a complex structure to be equivalent to the original one, Subsection 3.1,
and the second series conditions for an operator whose square is a perturbation of $-Id$ to be itself a perturbation of a complex structure. Finally we investigate relations between complex structures on a space and its hyperplanes, concluding with Theorem \ref{hypandspace}  (Theorem \ref{16} of the introduction).

In the fourth section, we pass to the context of twisted sums, defining and studying $\varepsilon$-singular twisted sums, obtaining restrictions on the existence of complex structures on such spaces (Proposition \ref{extconstante}) extending results from \cite{CCFM}.

In the fifth section, we investigate some symplectic structures on Banach spaces, including the natural symplectic structure on twisted Hilbert spaces induced by interpolation, the main example being Kalton-Peck space $Z_2$. Finally in the sixth section, we improve the description of the *-algebra homomorphism $\Lambda$ from $L(Z_2)$ into $L(H)$ defined by Kalton-Swanson, proving that it is isometric with respect to the $\|.\|_S$ seminorm on $L(Z_2)$ (Corollary 6.15), defining a similar map $\beta$ wich is easier to study. In the rest of the Section, we investigate extension properties of maps on $\ell_2$ to $Z_2$, the main technical result being the existence of a norm one operator $t_n$ on $\ell_2$ (which may be chosen to be a complex structure) which is extensible, but such that any extension to $Z_2$ of any compact perturbation of $t_n$ must have norm at least $n$ (Corollary \ref{inequiv}). Our main results then follow: a) $\|.\|_S$ is not complete on $L(Z_2)/S(Z_2)$ (Corollary \ref{nor}, or equivalently Theorem \ref{18}); b) the norm $\|.\|_S$ and the usual quotient norm are inequivalent on $L(Z_2)/S(Z_2)$ (Theorem \ref{630} or  Theorem \ref{17}); and c) the quotient norm
is not a $B^*$-algebra norm (Theorem \ref{631}, or Theorem \ref{19}).

\section{Quantitative notions of singularity for operators}
\subsection{Notions of singularity}
The following notions are well-known, although the terminology may vary. For references we indicate \cite{LT,Maurey}.
\begin{definition}
Let $S$ be an operator between Banach spaces $X$ and $Y$, and $\varepsilon>0$. We say that
\begin{itemize}
 \item[(a)]
$S$ is $\varepsilon$-singular if every infinite dimensional subspace of $X$ contains a normalized vector $x$ such that
$\|Sx\| \leq \varepsilon$.
\item[(b)]
$S$ is $\varepsilon$-strongly singular if for every infinite dimensional subspace of $X$, there is a further infinite dimensional subspace $Z$ such that $\|S_{|Z}\| \leq \varepsilon$. 

\item[(c)] $S$ is $\varepsilon$-infinitely singular if every finite codimensional subspace of $X$ contains a normalized vector $x$ such that
$\|Sx\| \leq \varepsilon$.
\end{itemize}
\end{definition}

We list a few obvious properties. For fixed $\varepsilon>0$, 
$\varepsilon$-strongly singular implies $\varepsilon$-singular, but there does not seem to be a kind of converse.
An operator is strictly singular when it is  $\varepsilon$-singular for all $\varepsilon>0$, and equivalently if it is $\varepsilon$-strongly singular for all $\varepsilon>0$.

By definition, an operator $S$ is infinitely singular if and only if it is $\varepsilon$-infinitely singular for all $\varepsilon > 0$, that is, if and only if for every $\varepsilon > 0$ there exists an infinite-dimensional subspace $Z$ of $X$ such that $\|S_{|Z}\| \leq \varepsilon$ \cite[Proposition 3.2, p.~15]{Maurey}.

\begin{proposition} \label{seminorm}
Let $X$ and $Y$ be Banach spaces with $\dim X=\infty$, and let $T$ be an operator in $L(X,Y)$. We use the following notation for the quotient norms by $S(X,Y)$ and $K(X,Y)$

$$\|T\|_{L(X,Y)/S(X,Y)}= \inf \{ \|T-S\| : S {\rm \ is \  strictly \ singular}\}$$

$$\|T\|_{\mathrm{Calkin}}=\inf \{ \|T-K\| : K {\rm \ is \ compact}\},$$
and define

$$\rho_s(T)=\inf\{\varepsilon: T {\rm\ is\ } \varepsilon-{\rm\ singular}\},$$

$$\|T\|_S=\inf\{\varepsilon: T {\rm\ is\ } \varepsilon-{\rm\ strongly \, \, singular}\}.$$
and
$$c(T)=\inf\{\varepsilon: T {\rm\ is\ } \varepsilon-{\rm\ infinitely\ singular}\}.$$
Then we have the following relations:
\begin{enumerate}[(a)]
\item $c(T) \leq\rho_s(T) \leq \|T\|_S \leq \|T\|_{L(X,Y)/S(X,Y)} \leq \|T\|_{\mathrm{Calkin}} \leq \|T\|$

\item $S$ is strictly singular $\Leftrightarrow
\rho_s(T + S)=\rho_s(T)$ for every $T\in L(X,Y)$ $\Leftrightarrow \|T+ S\|_S=\|T\|_S$ for every $T\in L(X,Y)$.

\item $\rho_s(T+U) \leq \rho_s(T)+\|U\|_S$, i.e. $\rho_s$ is $1$-Lipschitz with respect
to $\|.\|_S$. 

\item $\|T\|_S$ is a semi-norm, and this induces a (possibly non complete) norm on $L(X,Y) / S(X,Y)$; if $Y=X$ then this is an algebra norm, i.e.
 $\| TU\|_S \leq \|T\|_S \|U\|_S$
\item $\rho_s(TU) \leq \min(\rho_s(T) \|U\|_S,\|T\| \rho_s(U))$
\item  $\|T^{-1}\|^{-1} \leq \|T^{-1}\|^{-1}_S \leq \rho_S(T)$
\end{enumerate}

\end{proposition}

\begin{proof} (a) is an immediate consequence of the definitions.

Note that $$\||T\||_S=\sup_{Z \subseteq X}\inf_{W \subseteq Z}\|T_{|W}\|,$$
where $Z$
and $W$ are assumed infinite dimensional. In \cite{F} it was proved that the above defines a seminorm on $L(X,Y)$ and, when $X=Y$, that it is submultiplicative. This proves (d) and the second part of (b). The first part of (b) follows from (c), which is easy to check. Assertions (e) and (f) are left to the reader. 
\end{proof}

In general $\|T\|_{\mathrm{Calkin}}$, $\rho_s(T)$ and $\|T\|_S$ could differ \cite{manolo}. 
On the other hand, if $T$ is an operator on $\ell_p$, $1\leq p<\infty$, it is well-known that it is compact if and only if it is strictly singular, meaning that $\|T\|_{\mathrm{Calkin}}=0 \Leftrightarrow \rho_s(T)=0$. This extends to the following numerical estimate:

\begin{proposition} \label{equals} Let $1 \leq p <\infty$. For any $T \in L(\ell_p)$, we have the equality $\|T\|_{\mathrm{Calkin}}=\rho_s(T)=\|T\|_S$.
\end{proposition}

\begin{proof} Consider $P_N : \ell_p \to \ell_p$ the $N$-th projection associated to the canonical basis of $\ell_p$ and assume $\varepsilon < \|T\|_{\mathrm{Calkin}}\leq \| T- TP_N\| =  \|T_{|[e_k,k> N]}\|$. Therefore, we can find a normalized successive sequence $(w_n)$ such that $\|T(w_n)\| > \varepsilon$ for all $n$. Up to a small perturbation, we may assume that $Tw_n=v_n$ are successive.
Then for any normalized vector $w=\sum_i \lambda_i w_i$ in $W= [w_n]$, (i.e. $\lambda_i$'s are an $\ell_p$-norm one sequence), $\|Tw\|=\|\sum_i \lambda_i v_i\|\geq \epsilon$. So we have that $\rho_s(T)\geq \varepsilon$. \end{proof}

We shall see later on that this result extends to certain self-extensions of $\ell_p$. That  $\|\cdot\|_{\mathrm{Calkin}}$  and $\|\cdot\|_S$ coincide on $L(\ell_p)/K(\ell_p)$
also follows from Proposition \ref{seminorm}. Indeed, Meyer \cite{Meyer} proved that the Calkin algebra $L(\ell_p)/K(\ell_p)$ has a unique algebra norm for $1 \leq p < \infty$.

\subsection{Fredholm and inessential operators}
Recall that the classical Atkinson characterization of Fredholm operators states that an operator $T$ is Fredholm if and only if there exist operators $A$ and $B$ such that $AT - \Id$ and $TB - \Id$ are compact. In fact, we can choose such operators with $\|A\| = \|B\| = \|T_0^{-1}\|$, where $T_0$ is an isomorphism onto its image, defined on a closed subspace of $X$ with finite codimension. This means that Fredholm operators are precisely those operators whose image by the quotient map in the Calkin algebra is invertible.

\

We introduce the following new definitions.
\begin{definition}
Let $S$ be an operator between Banach spaces $X$ and $Y$, and $C>0$. We say that
\begin{itemize}
 \item[(a)] $S$ is $C$-inessential if  $\Id_Y-ST$ (equivalently, $\Id_X - TS$) is  Fredholm whenever $T \in L(Y,X)$ satisfies $\|T\| < C$.

 \item[(b)] $S$ is $C$-strongly inessential if $\Id_Y-ST$ (equivalently, $\Id_X - TS$) is  Fredholm whenever $T \in L(Y,X)$ satisfies $\|T\|_S < C$
    \end{itemize}
\end{definition}

Therefore $C$-strongly inessential implies $C$-inessential. Note also that the identity operator is always is $1$-inessential, but not $1+\varepsilon$-inessential, $\varepsilon>0$.
One may also observe that the sum of a $C$-inessential (resp. $C$-strongly inessential) operator with an inessential one is again $C$-inessential (resp. $C$-strongly inessential).

\begin{obs} An operator 
 $S$ is inessential if and only if it
 is $C$-inessential for all $C>0$, or equivalently,  $C$-strongly inessential for all $C>0$.
\end{obs}

\begin{lemma} \label{complexification}
Let $S \in L(X)$ be $C$-inessential. Then its complexification $\widehat{S} \in L(X_\C)$ is $C/\sqrt{2}$-inessential.
\end{lemma}

\begin{proof} 
Let $T \in L(X_\C)$ with $\|T\| < C/\sqrt{2}$. We need to show that $\Id_{X_\C} - T \widehat{S}$ is Fredholm. Write $T = A + iB$.
By writing $T(x+ix)=Ax-Bx+i(Ax+Bx)$, we see that
$\|A-B\| \leq \sup_{\|x\|=1} \|T(x+ix)\| \leq \|T\| \sqrt{2}<C$ and a similar estimate holds for $\|A+B\|$.
Therefore, by hypothesis, there exist operators $U_1, U_2 \in L(X)$ such that:

$$ U_1(\Id_X - (A + B)S) = \Id_X + K_1$$

$$ U_2(\Id_X - (A - B)S) = \Id_X + K_2,$$
where $K_1, K_2 \in  L(X)$ are compact operators. Now, consider the operator $U \in L(X_\C)$ defined by $U = \frac{1}{2}(U_1 + U_2) + i \frac{1}{2}(U_2 - U_1)$. Since:

$$ \Id_{X_\C} - T \widehat{S} = \Id_{X_\C} - AS + i(-BS), $$
we get:
$$ U(\Id_{X_\C} - AS + i(-BS)) = \Id + \frac{1}{2}(K_1 + K_2) + i \frac{1}{2}(K_2 - K_1) = \Id_{X_\C} + K, $$
where $K \in L(X_\C)$ is compact. Therefore, $\Id_{X_\C} - T \widehat{S}$ is Fredholm.
\end{proof}

We do not know whether an analogous result holds for $C$-strongly inessential, $\varepsilon$-singular, or $\varepsilon$-strongly singular operators.

\begin{proposition}\label{propo} Let $S \in L(X,Y)$.
    If $S$ is $\varepsilon$-singular, then $S$ is $1/\varepsilon$-strongly inessential.
\end{proposition}

\begin{proof}
This is essentially the proof of \cite[Proposition 
2.c.10]{LT}.
    Let $T$ an operator such that $\|T\|_S < 1/\varepsilon$. If $\dim \ker (\Id_X-ST)= \infty$, then there exists an infinite dimensional subspace $X'$ of $X$ such that $ST x= x$ for all $x\in X'$, and passing to a subspace we may also assume that
    $\|T_{|X'}\| <1/\varepsilon$. It follows that $T(X')$ is infinite dimensional and $\| S(Tz)\|\leq \varepsilon \|T_{|X'}\| \|z\|$ for some $z\in X'$, which is a contradiction.

    Assume that $\Id-ST$ has no closed range. Let $\|T\|_S<C<\frac{1}{\varepsilon}$.
    Then by \cite[Proposition 2.c.4]{LT} there exists an infinite dimensional subspace $X'$ of $X$ such that $\|(\Id_X -ST)_{|X'}\|\leq (1-\varepsilon C)/2$, and passing to a subspace we may also assume that
    $\|T_{|X'}\| < C$. Therefore, $ST_{|X'}$ is invertible, and there exists $z\in S_{X'}$ such that $\|S(Tz)\| \leq \varepsilon\|T_{|X'}\| \leq C $, again a  contradiction.
\end{proof}

Note that the converse of Proposition \ref{propo}  does not hold. Indeed the canonical inclusion $i$ of $\ell_2$ inside Kalton-Peck's space $Z_2$ is inessential but far from being singular.

\begin{lemma} \label{fredholm2}
Let $T: X \rightarrow Y$ be a Fredholm operator. Then there exists a constant $C(T) > 0$ such that if $S$ is $C$-inessential for some $C > C(T)$, we have $T + S$ is Fredholm, and $i(T + S) = i(T)$.
\end{lemma}

\begin{proof}
Let $U = T_0^{-1}$, where $T_0$ is the restriction of $T$ to a finite codimensional subspace such that $T_0$ is an isomorphism onto its image, and $T_0 X = TX$. Since $TX$ is complemented in $Y$, let $P: Y \to Y$ be a projection onto $TX$. Define $C(T) = \|U\| \|P\|$. If $S$ is $C$-inessential with $C > C(T)$, then $\Id + S(UP)$ is Fredholm.

Now, assume $\dim \ker (T + S) = \infty$. Then there exists an infinite-dimensional subspace $X' \subseteq X_0$ such that $(T + S)_{|X'} = 0$. Since $(T + S)x = (\Id + SUP)Tx$ for every $x \in X_0$, we have $\dim \ker (\Id + SUP) = \infty$, which is a contradiction.

If $T + S$ does not have closed range, then by \cite[Proposition 2.c.4]{LT}, there exists an infinite-dimensional subspace $X'$ of $X_0$ such that $(T + S)_{|X'}$ is compact. This again leads to a contradiction, since $(T + S)x = (\Id + SUP)Tx$ for every $x \in X'$ and $\Id + SUP$ is Fredholm.
\end{proof}

Since the canonical inclusion $i: \ell_2 \rightarrow Z_2$ is upper semi-Fredholm, but also inessential, the above lemma does not admit an immediate generalization to upper semi-Fredholm operators. However we do have classically:

\begin{proposition}\label{fredholm}
Let $T: X \rightarrow Y$ be an upper semi-Fredholm  operator. Then there exists 
$\varepsilon(T)>0$ such that whenever $S$ is $\varepsilon$-singular for some $\varepsilon<\varepsilon(T)$, then
$T+S$ is also upper semi-Fredholm and $i(T+S)=i(T)$.
\end{proposition}

Note that $\varepsilon(T)$ may be chosen equal to $\|T_0^{-1}\|^{-1}$ if $T_0$ is the restriction of $T$ to some finite codimensional subspace where it is an isomorphism onto its image.

\begin{proposition}  \label{spectrum}Let $S \in L(X)$.
 Assume $S$ is $1/\varepsilon$-inessential. 
 Then the spectrum of $S$ consists of a subset of the closed ball of radius $\varepsilon$ together with a sequence of eigenvalues of finite multiplicity which is either finite or such that the sequence of its moduli converges to 
 $\varepsilon$. 
\end{proposition}

\begin{proof} 
We observe that $\Id - \lambda S$ is Fredholm whenever $|\lambda| < 1/\varepsilon$, by $1/\varepsilon$-inessentiality, which rephrases as 
$S-\alpha \Id$ is Fredholm (with index 0 by continuity) whenever $|\alpha| > \varepsilon$. If such an $\alpha$ is an element of the spectrum, this  implies it is an eigenvalue with finite multiplicity which is isolated in the spectrum. This implies that the set of such values is finite or is countable with an accumulation point in the sphere of radius $\varepsilon$ (see the proof of Lemma 2.c.12 in \cite{LT}).
\end{proof} 

\section{Complex structures and ideals}

The aim of this section is to revisit results of \cite{FGEven} and \cite{CCFM} about inessential perturbations of complex structures; for example, by extending the condition that the perturbation has $\|.\|_S$-norm zero (i.e. is strictly singular) to the $\|.\|_S$-norm being small enough.
This will be done in the context of two complex structures on a space $X$, or of two complex structures on $X$ and one of its hyperplanes respectively.
\subsection{Perturbations of complex structures}

In what follows, let $X$ be a real  infinite dimensional Banach space.

\begin{lemma} 
If $J$ is a complex structure on $X$, then $\|J\|^{-1} \leq \|J\|_S \leq \|J\|$
\end{lemma}

\begin{proof}
From Proposition \ref{seminorm} (f).
\end{proof}

\begin{pro}
    Let $I$ be a complex structure on $X$. Let $J$ be a complex structure on $X$  such that  $J-I$ is $c$-inessential with $c>\|I\|/2$, or $C$-strongly inessential with
    $C>\|I\|_S/2$. Then $I$ and $J$ are equivalent.
    \end{pro}
 
 \begin{proof}  Let us call $s= J-I$. The operator $T= 2I+s$ is $\C$-linear from $X^I$ to $X^J$.
 Also, $T=2I(Id-Is/2 )$  by definition is Fredholm (necessarily of index $0$) as a real operator since $s$ is $(\|I\|/2)+\varepsilon$-inessential or
 $(\|I\|_S/2)+\varepsilon$-strongly inessential. Then  
 the operator $T$ is Fredholm of index 0 as a $\C$-linear map. It follows that $X^I$ and  $X^J$ are $\C$-isomorphic. \end{proof}

 \begin{pro}\label{hum}  Let $I, J$ be  complex structures on a space $X$. Then  $I$ and $J$ are equivalent whenever
 $(I-J)^2$ is $C$-inessential with $C>1/4$. In particular,  $I$ and $J$ are equivalent whenever 
 $\rho_s((I-J)^2)<4$.
 \end{pro}

\begin{proof}
 Note that $(I+J)^2+(I-J)^2=-4 \Id$ if $I$ and $J$ are complex structures.
 So $-\frac{1}{4}(I+J)^2=\Id+\frac{1}{4} (I-J)^2$.
 If $(I-J)^2$ is $C$-inessential for $C>1/4$, then $I+J$ is Fredholm with index $0$ and so $I$ and $J$ are equivalent.
\end{proof}
 \begin{corollary} Let $I, J$ be  complex structures on a space $X$. Then $I$ and $J$ are equivalent whenever $\rho_S(I-J)\|I-J\|_S <4$, and in particular, whenever $\|I-J\|_S<2$. 
 \end{corollary}

\begin{proof} A combination of Proposition \ref{hum} and Proposition \ref{seminorm}(e).
\end{proof}

\subsection{Operators whose square are close to $-\Id$} \label{perturbations}
The next series of results extends the case of inessential perturbations from \cite{FGEven}: the general principle is that if the square of an operator $I$ on $X$ is close enough to $-\Id_X$, then $I$ must be close to some complex structure.
In our first result, when proximity to $-\Id_X$ is relative to the norm, then we obtain closeness to a complex structure on the whole space $X$.
 
 \begin{pro} Let $I\in L (X)$ such that $\alpha=\|I^2 + \Id\|< 1$. Then there exists $S \in \mathcal L (X)$ belonging to the closed algebra generated by $I$ and $\Id$, with  $\| S\| \leq \|I\| (-1+ (1-2\alpha)^{-1/2})$ such that $(I+S)^2=-\Id$. 
 \end{pro}

\begin{proof}
Let us call $s= I^2+\Id$. 
The series $\sum_{n=1}^\infty b_n s^n$ converges to an operator $R$, where $\sum_{n=1}^\infty b_n z^n = -1 + (1-z)^{-1/2}$ for every $|z|<1$, and we observe that
$(I+IR)^2=-\Id$.
The result follows by taking $S=IR$, since $\|S\| \leq \|I\| \sum_{n=1}^\infty b_n (\|s\|)^n= \|I\| (-1+ (1-\alpha)^{-1/2})$.
\end{proof}

As a consequence, if a real space $X$ does not admit complex structures, then $T^2+\Id$ must have norm at least $1$ for any $T \in L(X)$. Examples of this are spaces with few operators, i.e. where every operator $T$ is a strictly singular perturbation of $\lambda \Id$, as in \cite{gowersmaurey}; then $T^2+\Id=(1+\lambda^2)Id+S$ has actually $\|.\|_S$-norm at least $1$. 

The previous proposition admits a version for $C$-inessential or $\varepsilon$-singular properties, but in this case, complex structures on hyperplanes must be taken into account as well, as in \cite[Proposition 10]{FGEven}:

\begin{pro} 
    Let $Y$ be an infinite dimensional real Banach space and let $X= \R \oplus Y$. Let $I \in L(X)$  such that $I^2 + \Id$ is $C$-inessential for some $C>\sqrt 2$. Then there exists an operator $S \in L(X)$ belonging, up to a finite rank perturbation, to the closed algebra generated by $I$ and $\Id$  such that either
    
    $$(I+S)^2=-\Id \, \, \text{or} \, \, I+S= \left[ {\begin{array}{cc}
   1 & 0 \\
   0 & J \\
  \end{array} } \right]$$ 
  where $J\in  L(Y)$ satisfies $J^2=-\Id$.

If furthermore $I^2+Id$ was $\varepsilon$-strongly singular,  $0<\varepsilon< 1$, then $S$ may be chosen to be $\|I\|_S (-1+(1-\varepsilon)^{-1/2})$-strongly singular.  
\end{pro}

\begin{proof}
    Let us call $s= I^2+\Id$. Going through the complexification, $\hat s$ is $C/\sqrt{2}$-inessential by Lemma \ref{complexification}. Let $\Gamma$ be a rectangular curve with horizontal and vertical edges, symmetric with respect to the horizontal axis, included in the region $\{z \in \C: \sqrt{2}/C<|z|<1\}$  and such that $\Gamma \cap Sp(\hat s)= \emptyset$.  
    Let $U$ and $V$ be the bounded and unbounded open regions delimited by $\Gamma$, respectively. Let $\widehat P$ and $\widehat Q$ be the spectral projections associated with $\Sp(\hat s) \cap U$ and $\Sp(\hat s) \cap V$. Then $\hat s = \hat s \widehat P + \hat s \widehat Q.$
    Since $\hat s \widehat P$ has spectral radius strictly less than $1$, following the argument of Ferenczi-Galego \cite[Proposition 10]{FGEven} as in the previous proposition,  consider
    $ R = \sum_{n=1}^\infty b_n ( s  P)^n$ and $\widehat R = \sum_{n=1}^\infty b_n (\hat s \widehat P)^n$, where $\sum_{n=1}^\infty b_n z^n = -1 + (1-z)^{-1/2}$ for every $|z|<1$ (note that by basic properties of functional calculus, the convergence of the series in the definition of $R$ in in norm, because the convergence radius of $\sum_{n=1}^\infty b_n z^n$ is $1$). It follows that $(\widehat I \widehat P+ \widehat I \widehat R)^2= - \widehat P$. Note that $IR= IsLP$, where $L= \sum_{n=1}^\infty b_n  s^{n-1}  $ and then $IR=IsL-IsLQ$.
    
    Therefore, under the assumption that $\|s\|_S \leq \varepsilon< 1$ and that $s$ is $C>\sqrt 2$-inessential 
    (which holds in particular if $\|s\|_S<\sqrt{2}/2$ by Proposition \ref{propo}), we deduce that 
    $$\|IR\|_S=\|IsL\|_S \leq \|I\|_S \sum_{n=1}^\infty b_n \|s\|_S^n=\|I\|_S(1-(1-\|s\|_S)^{-1/2}) \leq \|I\|_S (1-(1-\varepsilon)^{-1/2}).$$
Note that it follows from Proposition~\ref{spectrum} that $Q$ has finite rank. If $Q$ has even rank, the same argument by Ferenczi-Galego \cite[Proposition 10]{FGEven} yields a real finite-rank operator $r$ on $QX$ such that $(\widehat I \widehat Q + \hat r)^2 = -\widehat Q$. Therefore, $S = IR + r$ satisfies $(I + S)^2 = -\Id$, and  $\|S\|_S \leq \|I\|_S(1- (1-\varepsilon)^{-1/2})$ in case $\|s\|_S \leq \varepsilon<1$.

If $Q$ has odd rank, the argument by Ferenczi-Galego \cite[Proposition 10]{FGEven} also provides a finite-rank operator $r$ such that $I+(IR+r) =   \left[ {\begin{array}{cc}
   1 & 0 \\
   0 & J \\
  \end{array} } \right]$ corresponding to the decomposition $\R \oplus Y$ for some operator $J$ with $J^2=-\Id$.
\end{proof}

\begin{pro}
 Let $0<\varepsilon \leq 2/3$. Assume $J$ is a complex structure on $X$, and $S \in L(X)$ is 
 $\varepsilon^{-1}$-inessential.  Then all points of the spectrum of $J+S$ at distance  to $\{-i,i\}$ strictly bigger than $\alpha:=3\|J\|\varepsilon$  are eigenvalues with finite multiplicity.
\end{pro}

\begin{proof} 
Let $\lambda$ be in the spectrum of $J+S$. Since $(J-\lambda \Id)(J+\lambda \Id)=-1-\lambda^2$, if $\lambda \notin \{-i,i\}$ then
$J-\lambda \Id$ is invertible and
its inverse has norm at most $\frac{\|J\|+|\lambda|}{|1+\lambda^2|}$. If $\varepsilon<
\frac{|1+\lambda^2|}{\|J\|+|\lambda|}$ then
$J-\lambda \Id+S$ is still Fredholm with index $0$ and therefore $\lambda$ is an eigenvalue of $J+S$ with finite multiplicity.

Finally we note that if $|\lambda+i|>\alpha$ and $|\lambda-i|>\alpha$ (and wlog $Im(\lambda) \geq 0$), then
$|1+\lambda^2|>\alpha$.
If $|\lambda| \leq 2\|J\|$ then
it follows that
$\frac{|1+\lambda^2|}{\|J\|+|\lambda|} > \frac{\alpha}{3\|J\|}=\varepsilon$. If $|\lambda| > 2\|J\|$ then we have
$\frac{|1+\lambda^2|}{\|J\|+|\lambda|} > \frac{2(|\lambda|^2-1)}{3|\lambda|}  \geq \frac{2\|J\||\lambda|}{3\|\lambda|}\geq \frac{2}{3} \geq \varepsilon$.
\end{proof}

\subsection{Complex structures and hyperplanes}

\begin{lemma}\label{par} Let $I$ and $J$ be complex structures on a space $X$ and on $H$ hyperplane of $X$, respectively. 
Then the operator $T=I_{|H}+i_{H,X}J$ is a $\C$-linear operator from $H_J$ into $X_I$.\end{lemma}
\begin{proof}
Indeed for all $x \in H$, we have
$TJx=IJx-x$ and
$ITx=-x+IJx$.
\end{proof}

As a curious consequence we obtain:

\begin{proposition} Let $1<p<\infty$. Assume Kalton's space $Z_p$ is not isomorphic to its hyperplanes. If $I$ is a complex structure on $Z_p$ and $J$ is a complex structure on its hyperplane $H$, then the operator
$T=I_{|H}+i_{H,X}J \in L(H,X)$ is not Fredholm.
\end{proposition}

\begin{proof} By Lemma \ref{par}, the real operator $T$ would be Fredholm with even index. Therefore $H \oplus \R^{2m}$ and $Z_p \oplus \R^{2n}$ would be isomorphic for some integers $m,n$. Since $Z_p$ is isomorphic to its subspaces of even codimension, $Z_p$ would be isomorphic to $H$.
\end{proof}

We also have:

\begin{proposition}
    Assume $I$ is a complex structure on a space $X$
and $J$ a complex structure on its hyperplane $H$. Then there exist
$0<\lambda \leq 1$ such that $I_{|H}-\lambda J$ is infinitely singular,
and $0<\lambda \leq 1$ such that $J- \lambda I_{|H}$ is infinitely singular.
\end{proposition}

\begin{proof} Since $I i_{HX} -i_{HX}J \in L(H,X)$ is $\C$-linear from $H^{-J}$ to $X^I$, if it is Fredholm then it has even index. On the other hand $I i_{HX}$ is Fredholm with index $-1$.
So by continuity of the Fredholm index, there must be some $0<\lambda \leq 1$ for which the map
$U i_{hX}-\lambda i_{HX}J$ is not Fredholm. The same reasoning applies with $i_{HX}J$ instead of $I i_{HX}$.
\end{proof}

\begin{theorem}\label{hypandspace} Let $I$ be a complex structure on a space $X$ and $J$ a complex structure on a hyperplane $H$ complemented by a projection $p_H$. Then $(p_HIi_{H,X}-J)^2$ is not $C$-inessential if $C>1/4$. In particular, $\|I_{|H}-J\|_S \geq 2$. 
\end{theorem}

\begin{proof} If $p_H$ denotes a projection from $X$ onto $H$, and $r_H$ the rank $1$ projection $r_H=\Id-p_H$, a computation shows that
$$(p_H I i_{H,X} +J)^2+(p_H I i_{H,X} -J)^2=2p_H I p_H I i_{H,X}-2\Id_H=-4\Id_H+F,$$
onde $F:=-2p_HIr_HIi_{H,X}$ is an operator of rank at most $1$ on $H$. Note also that by Lemma \ref{par}, the operator $p_H Ii_{H,X}+J=p_H(I_{|H}+i_{H,X}J)$ from $H$ into $H$ cannot be Fredholm with even index.

Now if $(p_HIi_{H,X}-J)^2$ were $C$-inessential for $C>4$, then $\Id_H-\frac{1}{4}(p_HIi_{H,X}-J)^2$ would be Fredholm with index $0$ on $H$, therefore by the above equality $(p_H Ii_{H,X}+J)^2$ would be Fredholm with index $0$ and $p_H Ii_{H,X}+J$ would be Fredholm with index $0$, a contradiction.

It follows that $\rho_S((p_H I i_{H,X} -J)^2) \geq 4$, and in particular, since $\|p_H\|_S=1$ that 
$\rho_S(I_{|H}-i_{H,X}J) \|I_{|H}-i_{H,X}J\|_S \geq 4$ and that $\|I_{|H}-J\|_S \geq 2$.
\end{proof}

\begin{definition}
Let $X$ be a Banach space, $H$ be a given hyperplane of $X$ and $X=H \oplus \R e$ a fixed decomposition.
Let ${\mathcal I}(X)$
be the set of complex structures on $X$ and ${\mathcal J}(X)$ the set of operators on $X$ of the matricial form
$\begin{pmatrix} J & 0 \\ 0 & 1 \end{pmatrix}$, for $J$ a complex structure on $H$, and let $Z(X)$ be the set of operators $T$ on $X$ such that $T^2+Id$ is strictly singular.
\end{definition}

Let $\tilde{\mathcal I}(X)$,  $\tilde{\mathcal J}(X)$, and $\tilde{Z}(X)$ be the sets of elements of $L(X)/S(X)$ associated to ${\mathcal I}(X)$,  ${\mathcal J}(X)$, and $\tilde{Z}(X)$ respectively, and note that $\tilde{Z}(X)$
is the set of elements of square $-1$ in $L(X)/S(X)$.
By Ferenczi-Galego \cite{FGEven} Propositions 8 and 10, 
$\tilde{Z}(X)$ is the disjoint union of $\tilde{\mathcal I}(X)$ and $\tilde{\mathcal J}(X)$.
One can prove:

\begin{proposition} \label{distance} Let $X$ be an infinite dimensional Banach space. Then the set $\tilde{Z}(X)$ is closed in $L(X)/S(X)$ for $\|.\|_S$. Furthermore, if $H$ is a given hyperplane of $X$ as above, then $$d(\tilde{\mathcal I}(X),\tilde{\mathcal J}(X)):=\inf_{a \in \tilde{\mathcal I}(X), b \in \tilde{\mathcal J}(X)}\|a-b\|_S \geq 2$$ and therefore the sets
    $\tilde{\mathcal I}(X)$ and $\tilde{\mathcal J}(X)$ are clopen in $\tilde{Z}(X)$ for $\|.\|_S$.
\end{proposition}

\begin{proof} If $J_n \in Z(X)$ is a sequence of operators on $X$ tending to some $J$ for $\|.\|_S$,
then 
since $\|.\|_S$ is a algebra norm,
$J_n^2$ tends to $J^2$ for $\|.\|_S$.
Therefore $-\tilde{Id}=\tilde{J}^2$ and therefore $J^2+Id$ is strictly singular, which means that $J$ belongs to $Z(X)$. The assertion  $d(\tilde{\mathcal I}(X),\tilde{\mathcal J}(X)) \geq 2$ is Proposition \ref{hypandspace}. 
\end{proof}

\section{Quantified versions of singularity for quasi-linear maps} 

In this section, constants are relative to  the original quasi-norm associated with a quasi-linear map between Banach spaces. If one wishes to replace the quasi-norm with an equivalent norm, the estimates remain of course valid up to multiplicative constants.

\begin{proposition}[$\varepsilon$-singularity for $\Omega$]\label{Omegaepsilonsingular} Let $\Omega: Y \to X$ be a quasi-linear map and let $$ 0 \xlongrightarrow{} X \xlongrightarrow{i} X \oplus_\Omega Y \xlongrightarrow{q} Y\xlongrightarrow{} 0$$ be the corresponding twisted sum. The following are equivalent 
 \begin{itemize}
  \item[(a)] the quotient map $q$ is $\varepsilon$-singular;
  \item[(b)] there is no infinite dimensional subspace  $Z'\subseteq X \oplus_\Omega Y$ forming topological direct sum $i(X)\oplus Z'$ in $X \oplus_\Omega Y$ for which $Z'$ is $\varepsilon^{-1}$-complemented;
  \item[(c)] there is no infinite dimensional subspace $Y' \subseteq Y$ and linear map $L: Y'\to X$  such that $(\Omega-L)_{|Y'}$ has norm at most $\varepsilon^{-1}-1$;
  \item[(d)] for any infinite dimensional subspace  $Z'\subseteq X \oplus_\Omega Y$, $d(i(X),S_{Z'})<\varepsilon$.
  
 \end{itemize}
 When (a)(b)(c)(d) hold  we shall say that $\Omega$ is $\varepsilon$-singular, or that
 $X \oplus_\Omega Y$ is $\varepsilon$-singular.
\end{proposition}

\begin{proof}

We first prove  $(a) \Leftrightarrow (b)$: if $Z'$ witnesses that $q$ is not $\varepsilon$-singular, then
for all $z\in Z'$ and $x\in X$, $ \|i(x)+z\| \geq \|q(i(x)+z)\| \geq \varepsilon \|z\|$, so $Z'$ is $\varepsilon^{-1}$-complemented in $i(X) \oplus Z'$.
If $Z'$ is  $\varepsilon^{-1}$-complemented in $i(X) \oplus Z'$, then $\|q(z)\| =  d(z,i(X))= \|z+i(x_0)\| \geq \varepsilon  \|z\|$ for every $z\in Z'$.

Now we prove  $(b) \Leftrightarrow (c)$: if $L$ witnesses that (c) is false, then the subspace $\{(Ly,y), y \in Y'\}$ is $\varepsilon^{-1}$-complemented in its direct sum with $i(X)$. In fact, $\|(Ly,y)\|_\Omega = \|Ly - \Omega y\| + \|y\| \leq \varepsilon^{-1}\|y\| \leq \varepsilon^{-1} \|(x+Ly,y)\|_\Omega$ for every $x\in X$. 

if $Z'$ witnesses that (b) is false, then we may assume that there is some $Y' \subseteq Y$ and some linear $L$ such that
$Z'=\{(Ly,y), y \in Y'\}$.
This means an inequality of the form
$$\|Ly-\Omega y\|+\|y\| \leq \varepsilon^{-1} (\|x+Ly-\Omega y\|+\|y\|)$$
and so $\|\Omega-L\|$ is $\varepsilon^{-1}-1$ bounded on $Y'$.

Finally, we prove $(b) \Leftrightarrow (d)$: observe that $Z'$ is $\varepsilon^{-1}$-complemented on $i(X)\oplus Z'$ if and only if $\|z\| \leq  \varepsilon^{-1}  \|x + z\|$ for every $x\in i(X)$ and $z\in Z'$, which is equivalent to $d(i(X), S_{Z'})\geq \varepsilon$. 
\end{proof}

\begin{proposition} 
 Consider an $\varepsilon$-singular twisted sum
 $X \oplus_\Omega Y$, and let $Z$ be a Banach space. Let $t: X \rightarrow Z$
and let $T$ be an extension of $t$ to $X \oplus_\Omega Y$. Then 
$\rho_s(T) \leq \|t\|+2\varepsilon\|T\|$

\end{proposition}

\begin{proof}

Let $\delta$ be such that there exists $W \subseteq X \oplus_\Omega Y$ such that
$\|Tw\|
 \geq \delta \|w\|$ for all $w \in W$. 
 Then by Proposition \ref{Omegaepsilonsingular} there exist normalized vectors $x \in X, w \in W$ such that $d(x,w) \leq 2\varepsilon$.
 It follows that $$\delta \leq \|T(w)\| \leq \|T(x)\|-\|T(w-x)\| \leq \|t\|+ 2\|T\| \varepsilon.$$
\end{proof}

\begin{proposition} \label{extconstante}

Let  $Z=X \oplus_\Omega Y$ be  an $\varepsilon$-singular twisted sum.
 Let $j$ be a complex structure on $X$ which
 extends to a complex structure on $Z$ of norm at most $K$. 
 Let $H$ be a hyperplane of $Z$ containing $X$. Then any  complex structure on $H$ extending $j$ must have norm at least 
  $2\varepsilon^{-1}-K$.
\end{proposition}

 \begin{proof}
  
 Denote by $\pi: Z \rightarrow Y$ the quotient map, by $I$ the extension of $j$ to $Z$, of norm $K$,  and 
 by $J$ an extension of $j$ to $H$.
  By Proposition \ref{hypandspace}, 
  $I_{|H}-J$ is not $\alpha$-singular, for any
  $\alpha<2$, so let
  $W$ be a subspace of $H$ such that
  $\|(I-J)w\| \geq \alpha \|w\|$ for all $w \in W$.
Since $(I-J)_{|X}=0$ we deduce for all $w \in W$ and $x \in X$
$$\alpha \|w\| \leq \|(I-J)(w+x)\| \leq \|I_{|H}-J\| \|w+x\|,$$
therefore $\alpha \|w\| \leq \|I_{|H}-J\| \|\pi(w)\|$.
Since $\pi$ is $\varepsilon$-singular,  pick $w \in S_W$ with $\pi(w)$ of norm at most $\varepsilon$, obtaining 
$\alpha \leq \|I_{|H}-J\| \varepsilon$.
It follows that
$$2 \leq \varepsilon\|I_{|H}-J\| $$
and so $\|J| \geq 2\varepsilon^{-1}-\|I\|=2\varepsilon^{-1}-K$.\end{proof}

Observe that it always holds that $\rho_s(\pi) \leq \|\pi\|=1$, giving the estimate $\|J\| \geq 2-K$ in the above proof.
When $K=1$, i.e. $I$ is isometric, then the estimate becomes $\|J\| \geq 1$, which is sharp.

As a corollary of Proposition \ref{extconstante}, we recover one of the main results of \cite{CCFM}.

 \begin{corollary}[Castillo-Cuellar-Ferenczi-Moreno 2017]
 Consider an exact sequence
 $0 \rightarrow X \rightarrow Z \rightarrow Y \rightarrow 0$ for which the quotient map $\pi$ is strictly singular.
 Let $j$ be a complex structure on $X$ which
 extends to a complex structure on $Z$.  Then $j$ does not extend to a complex structure on an hyperplane of $Z$ containing $X$.
 \end{corollary}

 \begin{proof} A complex structure $J$ on a hyperplane $H$ extending $j$ would have norm larger than $2\varepsilon^{-1}-K$ for all $\varepsilon>0$. \end{proof}

\

Going now to full singularity,
it is well-known that if a twisted sum
$X \oplus_\Omega Y$ is singular, and $T: X \oplus_\Omega Y \rightarrow Z$, then $T$ is strictly singular if and only if $T_{|X}$ is strictly singular (the first explicit occurrence seems to be in \cite{Fhdn}, Lemma 2;  it also follows in a more homological language from \cite{CabelloCastillo} Lemma 9.1.5). This can be easily improved to the following quantitative estimates:

 \begin{proposition}\label{restriction} Assume a twisted sum
$X \oplus_\Omega Y$ is singular, and let $T: X \oplus_\Omega Y \rightarrow Z$.
Then $\rho_s(T)=\rho_s(T_{|X})$ and  $\|T\|_S=\|T_{|X}\|_S$.
\end{proposition}

\begin{proof} For the first affirmation, we only need to prove that
$\rho_s(T)\leq \rho_s(T_{|X})$, so assume $\delta>0$ and $W \subseteq X \oplus_\Omega Y$ are such that $\|Tw\| \geq \delta \|w\|$ for all $w \in W$. By classical results (\cite[Lemma 8]{CCK},\cite[Lemma 1]{Fhdn}) there exists a subspace $W'$ of $W$ and an embedding $\alpha: W' \rightarrow X$ such that $k=\alpha-i : W' \to X$ is compact of norm at most $\varepsilon>0$. Then for any $z=\alpha(w) \in \alpha(W') \subseteq X$, $w \in W'$, we have that
$T(z)=T(kw+w)=T(kw)+Tw$ has norm at least
$\delta \|w\|-\varepsilon\|T\|\|w\| \geq (1+\varepsilon)^{-1}(\delta-\varepsilon\|T\|)\|z\|$; therefore
$T_{|X}$ is not $(1+\varepsilon)^{-1}(\delta-\varepsilon\|T\|)$-singular.
Since $\varepsilon$ was arbitrary, we get the expected result. 

 The second affirmation is obtained in a similar manner: if $\|T\|_S>\delta$ and $W \subseteq X \oplus_\Omega Y$ is such that  every $W'\subseteq W$ contains a normalized vector $w\in W'$ with $\|Tw\|>\delta$, then there is an embedding $\alpha : W' \to X'$ of the form $i+k$, for some subspaces $W'\subseteq W$ and $X'\subseteq iX$ and a compact operator $k$ with $\|k\|<\varepsilon$. It follows that for every $X''\subseteq X'$, there is $x=\alpha(w)\in X''$ such that $\|Tx\|\geq \|Tw\|-\|Tkw\|> \delta - \|T\|\varepsilon$ and then $T_{|X}$ is not $\delta- \|T\|\varepsilon$-strongly singular.

\end{proof}

\begin{defi} If $Z$ is a twisted sum $X \oplus_\Omega Y$, we consider the restriction map
$$r: L(Z) \rightarrow L(X,Z)$$ and the induced map 
$$\tilde{r}: L(Z)/S(Z) \rightarrow L(X,Z)/S(X,Z).$$
\end{defi} 

\begin{remark}
By Proposition \ref{restriction}, if the twisted sum $Z=X \oplus_\Omega Y$ is singular, then the map $\tilde{r}$ is injective and is actually
 isometric with respect to the $\|.\|_S$-norms. 
 \end{remark}

 From the above we obtain the following principle to obtain inequivalent norms on a quotient $L(Z)/S(Z)$, which we shall apply in Kalton-Peck spaces in Chapter 6. 

\begin{proposition}\label{principle}
   Assume a twisted sum
$Z=X \oplus_\Omega Y$ is singular. Assume that for every $n \in \N$, there is a map $t_n \in L(X,Z)$ with norm at most $1$, which can be extended to a bounded map in $L(Z)$, but such that for any $s \in S(X,Z)$, any extension of $t_n+s$ has norm at least $n$. Then the quotient norm and the norm $\|.\|_S$ are inequivalent algebra norms on $L(Z)/S(Z)$.
\end{proposition}

 \begin{proof} If $T_n$ is any extension of $t_n$, then $\|T_n\|_S =\|t_n\|_S \leq 1$ by Proposition \ref{restriction}, while
 $\|T_n\|_{L(X)/S(X)} \geq n$.
 \end{proof}

\section{Symplectic structures and twisted Hilbert spaces induced by complex interpolation} \label{symplectic}

Recall that a real (respectively, complex) Banach space $X$ is called \emph{symplectic} if there exists a continuous alternating bilinear form $\omega : X \times X \to \mathbb{R}$ (respectively, a sesquilinear form $\omega : X \times X \to \mathbb{C}$, i.e.,  linear in the first variable and conjugate-linear in the second) such that $\omega(x,y) = -\omega(y,x)$ (respectively, $\omega(x,y) = -\overline{\omega(y,x)}$ for all $x,y \in X$), and the induced operator $L_\omega : X \to X^*, \, L_\omega(x)(y) = \omega(x,y)$, is an isomorphism onto $X^*$ (respectively, onto the conjugate dual $\overline{X}^*$). This condition implies that $X$ is reflexive, and we can identify $L_\omega^* = -L_\omega$ (respectively, $\overline{L_\omega}^* = -L_\omega$). 

In \cite{kaltonasymplectic}, Kalton and Swanson showed that the Kalton–Peck space $Z_2$ admits a symplectic structure defined by 
\[
\omega((x,y),(x',y')) = \langle x, y' \rangle_{\ell_2} - \langle y, x' \rangle_{\ell_2}, \qquad (x,y), (x',y') \in Z_2,
\]
They used this form to define, an involution $*$ on $L(Z_2)$, where for each operator $T \in L(Z_2)$, its adjoint operator $T^*$ is determined by 
\[
\omega(Tu,v) = \omega(u,T^*v), \qquad u,v \in Z_2.
\]
This construction also extends to the complex setting, by considering the complex version of $Z_2$.
More generally, symplectic structures arise naturally in the study of certain twisted Hilbert spaces obtained by complex interpolation. 

\begin{lemma}\label{lemo} Let $X$ be a reflexive complex Banach space such that $(X, \bar X^*)$ is a compatible couple and such that $X \cap \bar X^*$ is dense in both $X$ and $\bar X^*$. Then $(X, \bar X^*)_{1/2}$   is isometric to a Hilbert space and the respective derived space is symplectic. 
\end{lemma}

\begin{proof}The proof is consequence of the following facts.

1) The fact that the interpolated space $(X, \bar X^*)_{1/2}$  is isometric to a Hilbert space follows from \cite{Wat}.

2) Let $(X_0,X_1)$ be a compatible couple of complex Banach spaces. Then $Z=(X_0,X_1)_{1/2}$ and $W=(X_1,X_0)_{1/2}$ are equals with  same norms. Also $(x,y) \mapsto (x,-y)$ is a linear isometry from $dZ$ onto $dW$.

Indeed, $T: \mathcal F(X_0,X_1) \to \mathcal F(X_1, X_0)$ given by $Tf( z) = f(1-z)$ for all $z\in \overline{\mathbb S}$ is an isometry.  Hence if $B: Z \to  \mathcal F(X_0,X_1) $ is a section, then $C(x)(z) = B(x)(1-z)$ is a section from $W$ to $ \mathcal F(X_1, X_0)$. Also $d/dz B(x)(1-z)_{|1/2}=-B'(x)(1/2)$.

3) Let  $\overline{X}=(X_0,X_1)$ be a compatible couple of complex Banach spaces such that $\Delta (\overline{X})=X_0\cap X_1$ is dense in $X_0$ and in $X_1$  and that at least one of the spaces, $X_0$ or $X_1$, is reflexive. Then 

\begin{itemize}
\item (Calder\'on, 1964)$(X_0^*, X_1^*)$ is a compatible couple and for every $\theta \in (0,1)$ we have $(X_0^*, X_1^*)_\theta=(X_0,X_1)_\theta^*$ isometrically.

\item (Rochberg; Weiss, 1983) The following  diagram is commutative 

\begin{equation}\label{duality}\begin{CD}
0 @>>> X_\theta ^* @>>>  (dX_\theta)^* @>>> X_\theta ^*@>>>0  \\
@. @|  @VV{T}V @| \\
0@>>> X_\theta ^*@>>> d(X_\theta^*)  @>>> X_\theta ^*@>>>0
\end{CD}\end{equation}
where $T(x^*,y^*)(x,y)= \pin{x^*}{y}+ \pin{y^*}{x}$ for all $(x^*,y^*)\in d(X_\theta^*)$ and all $(x,y)\in dX_\theta$. 
\end{itemize}

Let $Z= d( (X, \bar X^*)_{1/2})$.  It follows that $S: Z\times Z \to \R $ given by  $ S(x,y)(z,w)= ( (x,y), (z,w)):= \pin{x}{w} - \pin{y}{z}$ is a symplectic structure on $Z$. 
\end{proof}
Under the conditions of Lemma \ref{lemo}, the algebra $L(Z)$ of all bounded operators on $Z$ is a *-algebra with the involution * defined by
\[ S(T\xi, \eta) = S(\xi, T^*\eta),\]
for all $\xi, \eta \in Z$. We denote by $L_S: Z\to Z^*$ the associated symplectic isomorphism given by $L_S (\xi)\eta= S(\xi, \eta)$. Hence $T^*= L_S^{-1} T^+ L_S$, where $T^+ \in L(Z^*)$ is the conjugate operator of $T$. 

\subsection{Perturbations of symplectic structures}

Analogously to the study of complex structures in Section \ref{perturbations}, in this section we extend some results of \cite{CCGP} concerning singular perturbations of symplectic structures to the case of inessential perturbations. Throughout this section, $X$ denotes a  real or complex reflexive infinite-dimensional Banach space.

\begin{pro} Let $\alpha: X \to X^*$ be an isomorphism such that $\|\alpha + \alpha^*\| < 2 \| \alpha^{-1}\|^{-1}$ ($\|\alpha + \overline{\alpha}^*\| < 2 \| \alpha^{-1}\|^{-1}$, for the complex case). Then $X$ admits a symplectic structure.
\end{pro}

\begin{proof} In the real case, let us call $\beta= \alpha - s/2$, where $s= \alpha + \alpha^*$. Then $\beta^*= - \beta$ and since 
\[ \|\beta - \alpha\|= \|s/2\| < \| \alpha^{-1}\|^{-1}\]
we have by the classical Neumann series argument for stability of invertibility under small perturbations,   that $\beta$ is an isomorphism.
\end{proof}

\begin{pro}
    Let $\alpha: X \to X^*$ be an isomorphism such that $\alpha + \alpha^*$ ($\alpha + \overline{\alpha}^*$, for the complex case) is $C$-inessential for some $C>2 \|\alpha^{-1}\|$. Then $X$   or its hyperplanes admit a symplectic structure.
\end{pro}
\begin{proof}
    In the real case, let us call $\beta= \alpha - s/2$, where $s= \alpha + \alpha^*$. Then $\beta^*= - \beta$ and by Lemma \ref{fredholm2} $\beta$ is a Fredholm operator with index 0. The rest of the proof follows the same arguments as in \cite[Proposition 8.7]{CCGP}.
\end{proof}

\begin{pro}
    Let $\alpha : X \to X^*$ be a symplectic isomorphism and $s: X \to X^*$ be $C$-inessential for some $C> 2 \|\alpha^{-1}\|$. If $\alpha+s$ is also a symplectic isomorphism, then $\alpha$ and $\alpha+ s$ are equivalent (i.e., there exists an isomorphism $T: X \to X$ such that $T^* \alpha T = \alpha+ s$.)
\end{pro}

\begin{proof}
     Let us call $S=\alpha^{-1} s$. Going through the complexification, $\widehat S$ is $ \sqrt 2C \| \alpha^{-1}\|^{-1}/2$-inessential. Consider $\Gamma$ a rectangular
with vertical and horizontal edges, rectifiable, conjugation-invariant, simple closed curve,
contained in the open unit disk, containing the open disk of radius $ \sqrt 2 \|\alpha^{-1}\|/C$, and such that $\Gamma \cap Sp(\widehat S) = \emptyset$.   Then Proposition \ref{spectrum} guarantees that all points of the spectrum of $\widehat S$ in the unbounded open domain $\mathbb{C} \setminus \Gamma$ form a finite set of eigenvalues of finite multiplicity. From this point, the rest of the proof follows the same argument as in \cite[Proposition~8.9]{CCGP}.

\end{proof}

\section{On Calkin algebras and homomorphisms into $L(H)$}

In this section the spaces considered (and in particular Kalton-Peck space) can again be either real or complex.
In Kalton-Swanson's paper \cite{kaltonasymplectic} 
we have the following remarkable result, where $H$ is a certain non-separable Hilbert space.

\begin{theorem}[Kalton-Swanson 1982]
There is a norm one *-algebra homomorphism $\Lambda$ from
$L(Z_2)$ into $L(H)$, sending $\Id$ to $\Id$. Furthermore, the kernel of $\Lambda$ is $S(Z_2)$.
\end{theorem}

The space $H$ is defined as the completion of the the quotient of $L(Z_2)$ by the subspace $Y=\{T: (T,T)=0\}$, where $(T,U)$ denotes a certain nontrivial symmetric  bilinear form on $L(Z_2)$ satisfying that
$(T,T)=\lim_{\U} \|T(e_n,0)\|_{Z_2}^2$.
For $T \in L(Z_2)$, the corresponding element in $H$ is denoted $\widetilde{T}$.

For $A \in L(Z_2)$, the map $\Lambda(A)$ is defined on $H$ by $$\Lambda(A)\widetilde{T}=\widetilde{AT}, T \in L(Z_2).$$

This defines a norm one algebra injective homomorphism $\hat{\Lambda}$
of $L(Z_2)/S(Z_2)$ into $L(H)$, and it was also proved by Kalton and Swanson that it is a *-homomorphism with respect to the symplectic structure on $Z_2$ defined in Section \ref{symplectic}. Our aim is to investigate further the properties of this map, as well as to extend it, and generalize it to other twisted sums.

\

Recall that a $B^*$-algebra is a Banach $*$-algebra whose norm satisfies the condition 
$\|a^* a\| = \|a\|^2$ for all $a \in A$ (\cite[pag 239]{Rickart}), while a $C^*$-algebra is defined as (*-isometric to) a closed *-subalgebra of $L(H)$ for some Hilbert space $H$.
In general, by the Gelfand-Naimark-Segal Theorem, one can construct  a Hilbert space $\mathcal H$ from a given {\em complex} $B^*$-algebra $A$ such that $A$ can be isometrically embedded into $L(\mathcal H)$ as a C*-subalgebra; so the two definitions coincide in the complex case. In particular the complex Calkin algebra $L(\ell_2)/K(\ell_2)$ is embedded as a C*-subalgebra of $L(\mathcal H)$ for some non-separable Hilbert space $\mathcal H$. 

In a similar manner we shall now see that $L(Z_2)/S(Z_2)$ (real or complex) can be embedded as a (non closed) *-subalgebra of $L(H)$, where $H$ is the non separable Hilbert space from \cite{kaltonasymplectic}, {\em if equipped with the $\|.\|_S$-norm}.

\subsection{Structures on self-extensions of $\ell_p$}

Given $T \in L(X,Y)$, recall that we denote
$\|T\|_{\mathrm{Calkin}}={\rm inf}_{K {\rm compact}}\|T+K\|$  the associated quotient norm in $L(X,Y)/K(X,Y)$, and that the inequality  $\|T\|_S \leq \|T\|_{\mathrm{Calkin}} \leq \|T\|_{L(X,Y)/S(X,Y)}$ holds.

In what follows, $\mathcal U$ denotes a non trivial ultrafilter on the integers.

\begin{proposition} \label{equalities}

Let $1 \leq p < \infty$, let $Y$ be a Banach space, and let $T \in L(\ell_p, Y)$. Define:
\begin{enumerate}[(a)]
    \item  
    $\displaystyle 
    \alpha(T) = \inf\{\, \|T_{|Z}\| : Z \subset \ell_p \text{ has finite codimension}\,\};$
    
    \item  
    $\displaystyle 
    \beta(T) = \sup\bigl\{\, \lim_{\U} \|T w_n\| : (w_n)\ \text{is weakly null in } \ell_p,\ 
           \lim_{\U} \|w_n\| = 1 \bigr\};$
    
    \item  
    $\displaystyle 
    \gamma(T) = \lim_{n \to \infty} \|T_{|X_n}\|,$
    where $X_n$ denotes the closed linear span of $\{e_k : k \geq n\}$.

\end{enumerate}
Then
\[
\alpha(T) = \beta(T) = \gamma(T) = \|T\|_{\mathrm{Calkin}}.
\]
\end{proposition}

\begin{proof} The equality $\alpha(T)=  \|T\|_{\mathrm{Calkin}}$ is true for any operator $T: X \to Y$ between Banach spaces. In fact, let $K: X \to Y$ be compact and $\epsilon >0$. Then there exist $y_1, \ldots , y_n$ in $Y$ such that $\min_k  \|Kx-y_k\| < \epsilon$ for every $x\in B_X$. Consider normalized $y_k^*\in Y^*$ such that $y_k^*(y_k)=\|y_k\|$, and let $Z = \cap_{k=1}^n \ker K^*y^*_k$. Now fix $x\in B_Z$, and let $y_k$ such that $\|Kx-y_k\| <\epsilon$. It follows that $\|y_k\| = \| y_k^*( y_k - Kx) \|<\epsilon$. Hence, $\|Tx\|\leq \|Tx-Kx\| + \|Kx-y_k\|+\|y_k\|<\|T-K\|+2\epsilon$. We conclude that $\alpha(T) \ \leq  \|T\|_{\mathrm{Calkin}}$.  Now if $Z\subseteq X$ has finite codimension, let $X=Z\oplus F$, where $F$ has finite dimension. Consider $P: X \to F$ a projection, then $K=TP$ is compact and $\|T\|_{\mathrm{Calkin}} \leq \|T-K\|=\|T|_Z\|$.

To establish that $\gamma(T)= \alpha(T) $, let $Z$ be a finite codimensional subspace of $\ell_p$. There exist $f_1, \ldots, f_n \in \ell_p^* = \ell_q$ such that $Z= \cap_{i=1}^n \ker f_i$. By density, for every $\epsilon >0$,  we obtain finite supported vectors $g_1, \ldots , g_n$ such that $\|f_i-g_i\|$ is small, set $W =  \cap_{i=1}^n \ker g_i$. Then $X_k \subseteq W$ for some $k$. A perturbation estimates gives $\|T_{|X_k}\| \leq \| T_{|W}\| \leq \|T_Z\| + \epsilon$. It follows that  $\lim_n \|T_{|X_n}\| \leq \alpha(T) $.

Now we prove that $\beta(T) = \|T\|_{\mathrm{Calkin}}$. Fix $(w_n)$ a weakly null sequence in $\ell_p$ with  $\lim_{\U} \|w_n\| = 1$. For every $K\in K(\ell_p, Y)$, $\lim_{\U} \| Tw_n\| =\lim_{\U} \| (T-K)w_n\| \leq\|T-K\|$. Therefore $\beta(T) \leq \|T\|_{\mathrm{Calkin}}$. For the other inequality,  fix $\epsilon> 0$ and let $w_n \in X_n$ normalized with finite support such that $\|T_{|X_n}\|-\epsilon < \| Tw_n\|$. Passing to a subsequence, we may assume that $(w_n)$ are successive, and then $\lim_n \|T_{|X_n}\| \leq \lim_{\U} \|Tw_n\| $.
 
 \end{proof}

The following extends Proposition \ref{equals}

\begin{proposition}\label{rhoequalnorm} Let $Y$ be a singular self-extension of $\ell_p, 1<p<\infty$. Let $T \in L(\ell_p,Y)$. Then $$\rho_S(T)=\|T\|_S=\|T\|_{\mathrm{Calkin}}.$$
\end{proposition}

\begin{proof}
If $\varepsilon< \|T\|_{\mathrm{Calkin}}  = \lim_n \|T_{[e_k, k \geq n]}\|$, then we can find $w_n$ normalized successive 
with $\|Tw_n\| \geq \varepsilon$ for every $n$. Let $W=[w_n]_n$.
Since $w_n$ tend weakly to $0$ and $\pi T \in L(\ell_p)$ is compact, we have $\pi Tw_n=d(Tw_n,\ell_p)$ tending to $0$. So passing to a subsequence and
up to a compact perturbation on $T$, we may assume that $Tw_n=v_n$ belongs to $\ell_p$, and then that $v_n$ is successive. 
Then for any $w=\sum_i \lambda_i w_i$ in $S_W$, (i.e. $\lambda_i$'s are an $\ell_p$-norm one sequence), $Tw=\sum_i \lambda_i v_i$ will be of norm at least $\varepsilon$.
So we have that $\rho_s(T) \geq \varepsilon$.
\end{proof}

\begin{corollary}\label{ssss} Let $Y$ be a singular self-extension of $\ell_p, 1<p<\infty$. Let $T \in L(Y)$. Then $\rho_s(T)=\|T\|_S=\|T_{|\ell_p}\|_{\mathrm{Calkin}}$. In particular, $T$ is strictly singular if and only if $T_{|\ell_p}$ is compact.
\end{corollary}

\begin{proof} This follows as a consequence of Proposition \ref{restriction} and Proposition \ref{rhoequalnorm}. \end{proof}

As an example, if $i$ denotes the embedding of $\ell_p$ inside $Y$ and $\pi$ the quotient map, then $T:=i\pi: Y \rightarrow Y$ is strictly singular non compact, but $T_{|\ell_p}=0$.

\

When $Y=Z_2$, Corollary \ref{ssss} was a well-known result, see \cite{CGP} for another proof as well as extensions to the higher dimensional Rochberg spaces; the same argument holds also in $Z_p, 1<p<\infty$:

\begin{corollary}[\cite{CGP}]  Let $T \in L(Z_p), 1<p<\infty$. Then 
$T$ is strictly singular if and only if
$T_{|\ell_p}$ is compact. 
\end{corollary}

\begin{remark} It follows that if
     $Y$ is a singular self-extension of $\ell_p$, and $A \in L(Y)$ with $\|A\|_S <\varepsilon$, then $A_{|\ell_p}=F+\varepsilon R$, where $F$ has finite rank and $\|R\| \leq 1$. \end{remark}

Indeed from $\|A\|_S=\|A_{|\ell_p}\|_S < \varepsilon$, we have $\|A_{\ell_p}\|_{\mathrm{Calkin}} < \varepsilon$, therefore $A_{|\ell_p}=F+\varepsilon R$ with $\|R\| \leq 1$.

\begin{corollary} \label{isomr} Let $Y$ be a singular self-extension of $\ell_p$, $1<p<\infty$. We consider the restriction map $r_p: L(Y) \to L(\ell_p, Y)$. Then the following is an isometric embedding induced by $r_p$:
 $$\widetilde{r_p}:  (L(Y)/S(Y),\|.\|_S) \rightarrow (L(\ell_p,Y)/ K(\ell_p,Y),\|.\|_{\mathrm{Calkin}})$$
$$\widetilde{T} \mapsto \widetilde{T_{|\ell_p}}$$
\end{corollary}

\subsection{Representation of operators on singular self-extensions of $\ell_p$}

We inspire ourselves from  the construction of the Hilbert space $H$ in \cite{kaltonasymplectic} to define a simpler version of the map of Kalton-Swanson, valid for all values of $1<p<\infty$:

\begin{definition}
Let $1 \leq p <\infty$ and $W_p \subseteq (\ell_p)_{\mathcal U}$ be defined by 
$$\widetilde{(x_n)_n} \in W_p \Leftrightarrow
w-\lim x_n=0.$$
\end{definition}
The space $W_p$ is a closed subspace of $(\ell_p)_{\mathcal U}$, and therefore in particular $W:=W_2$ is a Hilbert space.

\begin{proposition} \label{normbeta} Let $Y$ be a singular self-extension of $\ell_p, 1<p<\infty$. 
Let $\alpha_p: L(\ell_p,Y) \rightarrow L(W_p)$ be defined by
$$\alpha_p(T)(\widetilde{(x_n)_n})=\widetilde{(Tx_n)_n}
\in Y_{\mathcal U},$$ 
and $\beta_p=\alpha_p \circ r_p: L(Y) \rightarrow L(W_p)$, i.e. 
$$\beta_p(T)(\widetilde{(x_n)_n})=\widetilde{(Tx_n)_n}
\in Y_{\mathcal U},$$

Then
\begin{enumerate}[(a)]
    \item $\alpha_p$ and $\beta_p$ are well-defined, i.e. for all $T: \ell_p \rightarrow Y$, $\alpha_p(T)$  takes values in $W_p \subseteq (\ell_p)_{\mathcal U} \subseteq Y_{\mathcal U}$.
    \item  $\alpha_p$ induces an isometric embedding $\widetilde{\alpha_p}$ of
$(L(\ell_p,Y)/K(\ell_p,Y), \|.\|_{\mathrm{Calkin}})$ into $L(W_p)$.
\item  $\beta_p$ induces an isometric embedding $\widetilde{\beta_p}=\widetilde{\alpha_p} \circ \widetilde{r_p}$ of $(L(Y)/S(Y),\|.\|_S)$ into $L(W_p)$.
\end{enumerate}

\begin{proof}
$(a)$ It is clear that $Tx_n$ tends weakly to $0$ in $Y$. If $\pi$ denotes the quotient map then $\pi T \in L(\ell_p)$ is strictly singular, therefore compact, so $\pi Tx_n$ tends in norm to $0$. Let therefore $y_n \in \ell_p$ be such that $d(Tx_n,y_n)$ tends to $0$. We have
$\widetilde{(Tx_n)_n}=\widetilde{(y_n)_n}$ belonging to $W_p$.
$(b)$ We have
$\|\alpha(T)\|=\|T\|_{\mathrm{Calkin}}$ by Proposition \ref{equalities}.
$(c)$ It follows from $(b)$ and Corollary \ref{isomr}.

\end{proof}

\end{proposition}

\begin{lemma}\label{Zp}Let $1\leq p <\infty$, let $T \in L(Z_p)$, and let $(x_n)$ be a sequence in $\ell_p$ tending weakly to $0$.  Then
$$\lim_{n \in \N} \|i \pi T(\Omega_p x_n,x_n)-Tx_n\|=0.$$

\end{lemma}

\begin{proof} 
We may first consider the case in which  $(x_n)$ is successive. Since $Z_p$ is singular, the operator $\pi T_{|\ell_p} : \ell_p \to \ell_p$ is strictly singular and hence compact. Therefore, $\lim \| \pi T x_n\| \to 0$. By passing to subsequences, we may assume that $\sum \|Tx_n - y_n\| < \infty$ for some semi-normalized successive sequence $(y_n)$ in $\ell_p$. Let $W$ and $U$ be the block operators associated with the sequences $(x_n)$ and $(y_n)$, respectively, as defined in \cite{kaltonasymplectic}, p.~389. While the construction is presented for $Z_2$, the same argument shows that these operators define isometries on $Z_p$ with respect to the quasi-norm. Define $K: Z_p \to Z_p$ by  $Ke_n= Tx_n - y_n$ and $K(0,e_n)= 0$. Then $K$ is compact and $(TW- U - K)_{|\ell_p}=0$; it follows that 
$TW- U - K$ factors through $\pi$ and is therefore strictly singular, and hence $TW-U$ is strictly singular as well. By  \cite[Theorem 5.4]{KaltonPeck} (see also \cite[Lemma 10.8.3]{CabelloCastillo}), we have $\lim \|\pi (TW-U)(0,e_n)\|=0$, for otherwise there would exist a subsequence  of $ (TW-U)(0,e_n)$ equivalent to the usual basis of the Orlicz sequence space $\ell_f$, where $f(t)= |t(-\log t)|^p$  in a neighborhood of 0. On the other hand, $$\| i\pi(TW-U)(0,e_n)\|_p= \| |\pi (TW-U)(0,e_n) \|= \|i\pi T(\Omega_p x_n, x_n) - y_n\|,$$ which establishes the result for this case.

For the general case where $(x_n)$ is weakly null, we use the following perturbation argument: by passing to a subsequence we may assume that  
$x_n = w_n + \delta_n f_n$,  for every  $n\in \N$, where $(w_n)$ and $(f_n)$ are normalized sequences such that  $(w_n)_n$ is successive,  $\mathrm{supp} w_n \cap \mathrm{supp} f_n = \emptyset$, and $\lim_{n \to \infty} \delta_n = 0$.  
Then we have  
\[
\|(\Omega_p x_n, x_n) - (\Omega_p w_n, w_n)\|
  = \|\Omega_p(w_n + \delta_n f_n) - \Omega_p w_n - \Omega_p(\delta_n f_n)\| + |\delta_n|.
\]
Moreover,  
\[
\|\Omega_p(w_n + \delta_n f_n) - \Omega_p w_n - \Omega_p(\delta_n f_n)\|
= 
\Big\| 
\log\!\left(\frac{1}{(1+\delta_n^p)^{1/p}}\right) w_n
+ \delta_n \log\!\left(\frac{\delta_n}{(1+\delta_n^p)^{1/p}}\right) f_n
\Big\|_p +\, |\delta_n|=
\]
\[
\left(
\left( -\frac{1}{p}\log(1+\delta_n^p)\right)^p
+ \left(\delta_n \log\tfrac{\delta_n}{(1+\delta_n^p)^{1/p}}\right)^p
\right)^{1/p} +\, |\delta_n|.
\]
which tends to $0$ since $\delta_n \to 0$.
\end{proof}

 Let $1<p,q<\infty$ with $\frac{1}{p}+\frac{1}{q}=1$. If $Y$ is a singular self-extension of $\ell_p$,  $(y_n)_n$ is a sequence in $Y$ such that $d(y_n,\ell_p)$ tends to $0$, and $(z_n)_n$ is a sequence in $\ell_q$, then we define
$<(y_n)_n,(z_n)_n>$ in the obvious way: as 
$\lim_{\mathcal U}<r_n,z_n>$,  where $(r_n)_n$ is a sequence in  $\ell_p$ such that $d(r_n,y_n)$ tends to $0$.

\begin{proposition}\label{proppq}  Let $1<p,q<\infty$ with $\frac{1}{p}+\frac{1}{q}=1$. If $T \in L(Z_p)$, $(w_n)_n \in \ell_q$,
and $(w'_n)_n \in \ell_p$ are weakly null sequences, then
$$ <(T^*w_n)_n, (w'_n)_n>= <(w_n)_n, (Tw'_n)_n>.$$
\end{proposition}

\begin{proof}
By Lemma \ref{Zp} we know  $$<(T^*w_n)_n, (w'_n)_n>=\lim_{\mathcal U} <i\pi T^*(\Omega w_n, w_n), w'_n>.$$
Note the identity
$<i\pi (x,y), x'>=<y,x'>=-B( (x,y), (x',0) )$, therefore
$$<(T^*w_n)_n, (w'_n)_n>=-\lim_{\mathcal U} B( T^*(\Omega w_n,w_n), (w'_n,0))
= -\lim_{\mathcal U} B( (\Omega w_n, w_n), Tw'_n).$$ 
Picking $y_n \in \ell_p$ such that $d( Tw'_n, y_n)$ tends to $0$, this is equal to $$
-\lim_{\mathcal U} B( (\Omega w_n, w_n), (y_n,0))=\lim_{\mathcal U} <w_n, y_n>=<(w_n)_n, (Tw'_n)_n>.$$
 
Therefore $ <(T^*w_n)_n, (w'_n)_n>= <(w_n)_n, (Tw'_n)_n>$.
\end{proof}

We therefore obtain a more explicit version of Kalton-Swanson map $\Lambda$:
 \begin{proposition} Let $1<p,q<\infty$ with $\frac{1}{p}+\frac{1}{q}=1$. Then the map $$\beta_p: (L(Z_p),\|.\|_S)
 \rightarrow L(W_p),$$ defined by $$\beta_p(T)(\widetilde{(w_n)_n})=\widetilde{(Tw_n)_n}$$ is isometric. Furthermore for all $T \in L(Z_p)$,
 $$\beta_p(T)^*=\beta_q(T^*).$$\end{proposition}
 \begin{proof} That $\beta_p$ is isometric was proved in Proposition \ref{normbeta}. If $w=\widetilde{(w_n)_n} \in W_q$ and 
 $w'=\widetilde{(w_n^{\prime})_n}\in W_p$, to check the *-homomorphic property we need to compute
 $$< \beta_q(T^*) \circ w,w'>
 =<\widetilde{(T^*w_n)_n},\widetilde{(w'_n)_n}>=\lim_{\U} <T^*w_n, w'_n>.$$
 On the other hand
 $$<   w, \beta_p (T) \circ w'>
 =<\widetilde{(w_n)_n},\widetilde{(Tw'_n)_n}>=\lim_{\U} <w_n, Tw'_n>.$$
 Therefore Proposition \ref{proppq} implies that $\beta_p(T)^*=\beta_q(T^*)$.
 \end{proof}

\begin{definition} We denote by $\beta$ the map $\beta_2$, i.e., $$\beta: (L(Z_2),\|.\|_S) \rightarrow L(W)$$ is the *-isometric map defined by 
$$\beta(T)(\widetilde{(w_n)_n)}=\widetilde{(Tw_n)_n}.$$
We denote by $\widetilde{\beta}$ the induced *-isometric embedding from
$(L(Z_2),\|.\|_S)$ into $L(W)$.
\end{definition} 

\subsection{Relation with the Kalton-Swanson map $\Lambda$}

Recall that  $\alpha: L(\ell_2,Z_2) \rightarrow L(W)$ is defined by
$\alpha(T)(\widetilde{(x_n)_n})=\widetilde{(Tx_n)_n}
\in W$ and that $r: Z_2 \to L(\ell_2, Z_2)$ denotes the restiction map, defined by $r(T)=T_{|\ell_2}$.
\begin{proposition}   
For each $T \in L(Z_2)$, define $j(T):=\widetilde{(T(e_n,0))_n}
\in W$. Let $\tilde{j}:H \rightarrow W$ denote the associated isometric embedding. Then for all $T \in L(Z_2)$,
$$\beta(T) \circ \tilde{j}=\tilde{j} \circ \Lambda(T)$$ 
\end{proposition}

\begin{proof}

Recalling that $H=L(Z_2)/\{(T,T)=0\}$ equipped with the seminorm associated to $(T,T)=\lim_{\U}\|T(e_n,0)\|_{Z_2}^2$, we see that $j$ indeed induces an isometric embedding $\tilde{j}$ of $H$ into $W$.
Let $A \in L(Z_2)$, then $\beta(A) \in L(W)$ is defined as
$$\beta(A)(\widetilde{(x_n)_n)}=\widetilde{(Ax_n)_n}.$$
In particular if $T \in L(Z_2)$ then
$$\beta(A)(\widetilde{(T(e_n,0))_n})=\widetilde{(AT(e_n,0))_n},$$
i.e.
which means that $\beta(A)$ acts from $jH$ into $jH$ and that
 $\beta(A) \circ \tilde{j} =\alpha\circ r(A) \circ \tilde{j}=\tilde{j} \circ \Lambda(A)$.

\end{proof}

We therefore improve on Kalton-Swanson's result by proving that $\Lambda$ does not only have kernel equal to $S(Z_2)$, but actually induces an isometric homomorphism on $L(Z_2)/S(Z_2)$ with respect to the $\|.\|_S$-norm.

\begin{corollary}\label{KS} The Kalton-Swanson map $\Lambda$ induces an isometric *-algebra homomorphism $\hat{\Lambda}$ from $(L(Z_2)/S(Z_2),\|.\|_S)$ into
$(L(H),\|.\|)$.
\end{corollary}

\subsection{Completions and closures in $L(Z_2)/S(Z_2)$}

Kalton and Swanson asked whether $\Lambda(L(Z_2))$ is closed in $L(H)$, \cite{kaltonasymplectic} Remark p392; this is equivalent to  asking whether $\beta(L(Z_2))$ is closed in $L(W)$. We shall answer this question by the negative, by studying extension properties of operators on $\ell_2$ to $Z_2$.

\begin{definition}
  If $Y$ is a self-extension of $\ell_p, 1<p<\infty$, let
 $$ext(\ell_p,Y)=r(L(Y))=\{ t \in L(\ell_p,Y): t {\rm\ admits\ an\ extension\ }T \in L(Y)\}$$  and let
 $$ext_K(\ell_p,Y)=ext(\ell_p,Y)+K(\ell_p,Y).$$
  
\end{definition}

\begin{lemma}\label{blabla} Let $Y$ be a self-extension of $\ell_p, 1<p <+\infty$. Then
\begin{enumerate}[(a)]
\item We have the inclusions $ext(\ell_p,Y) \subseteq ext_K(\ell_p,Y) \subseteq clos(ext(\ell_p,Y)).$ 
\item The closures of $ext(\ell_p,Y)$ and of $ext_K(\ell_p,Y)$ are equal.

\item $ext(\ell_p,Y)$ is closed if and only if it is equal to $L(\ell_p,Y)$.
\end{enumerate}\end{lemma}

\begin{proof} (a) Since $\ell_p$ has the Approximation property, and since finite rank operators are always extendable, we have
$$K(\ell_p,Y) \subseteq clos(F(\ell_p,Y))
\subseteq clos(ext(\ell_p,Y)).$$
Therefore the second inclusion holds. (b) follows immediately.
(c) If $ext(\ell_p,Y)$ is closed, applying AP we then have that all compact operators are extendable; by \cite{UFObis} this implies that $(\ell_p,Y)$ is a UFO-pair \cite{UFO}, and by reflexivity of $Y$, this implies that $(\ell_p,Y)$ is an extensible pair (\cite{UFO} Lemma 1.2), meaning that all operators on $\ell_p$ are extendable. 
\end{proof}

\begin{proposition} Assume $1<p<\infty$ and $Y$ is a singular self-extension of $\ell_p$.
We have that 
$$clos(r(L(Y)))=clos(ext(\ell_p,Y))=clos(ext_K(\ell_p,Y)),$$
where the closures are $\|.\|$-closures in $L(\ell_p,Y)$.

\end{proposition}

\begin{proof} The proof comes from the definitions and Lemma \ref{blabla}

\end{proof}

\begin{proposition}\label{equivalence} Let $Y$ be a singular self-extension of  $\ell_p, 1<p<\infty$. The following are equivalent
\begin{enumerate}[(a)]
\item (when $Y=Z_2$)  $\Lambda(L(Z_2))$ is closed in $L(H)$
\item $\beta(L(Y))$ is closed in $L(W_p)$
\item $L(Y)/S(Y)$ is $\|.\|_S$-complete
\item $\|.\|_S$ is equivalent to the quotient norm on $L(Y)/S(Y)$
\item $r(L(Y)/S(Y))$ is closed in $(L(\ell_p,Y)/K(\ell_p,Y),\|.\|_{\mathrm{Calkin}})$
\item $ext_K(\ell_p,Y)$ is norm closed in $L(\ell_p,Y)$
\end{enumerate}
\end{proposition}

\begin{proof} The equivalence of (a)(b)(c) is clear, since
$\beta$ (and $\Lambda$) both induce an isometric map between $(L(Y)/S(Y),\|.\|_S)$ and the respective images; same for (e). The equivalence with (d) is clear since $\|.\|_S$ is always dominated by the quotient norm, which is complete.

(f) $\Rightarrow$ (c): assume we have an absolutely converging series
$\sum_n \tilde{T_n}$ in $(L(Y)/S(Y),\|.\|_S)$, and
denote $t_n=(T_n)_{|\ell_p}$. Since we may replace $T_n$ by any finite rank perturbation, we may replace $t_n$ by any finite rank perturbation, and since $\ell_p$ has AP we may wlog assume $\|t_n\| \leq 2 \|t_n\|_{\mathrm{Calkin}}=2\|T_n\|_S$. Then $\sum_n (T_n)_{|\ell_p}=\sum_n t_n$ converges in norm, to a map $s \in L(\ell_p,Y)$.
If (f) holds, then the limit map $s$ belongs to $ext_K(\ell_p,Y)$, i.e.  $s+k$ is extendable to a map $S \in L(Y)$ for some compact map $k$. Therefore we have 
that $\sum \tilde{T_n}$ converges (to $\tilde{S}$):
indeed $\|\sum_{n=1}^N \tilde{T_n}-\tilde{S}\|_S=
\|\sum_{n=1}^N t_n-s-k\|_{\mathrm{Calkin}} \leq \|\sum_{n=1}^N t_n-s\|$,
which converges to $0$. 

(d) $\Rightarrow$ (f):
Let $\sum_n t_n$ be  absolutely converging  in $ext_K(\ell_p,Y)$, and let $t$ be its limit.
Let $T_n \in L(Y)$ be such that $k_n=(T_n)_{|\ell_p}-t_n$ is compact.
We have that $\sum_n \|\tilde{t_n}\|_{\mathrm{Calkin}}$ converges so $\sum \|T_n\|_S$ converges and therefore, by (d), this holds if one replaces $\|.\|_S$ by  the quotient norm on $L(Y)$.  We may replace $T_n$ by a strictly singular perturbation so that
$\sum_n \|T_n\|$ converges.
Then $\sum_n T_n$ converges to some $T$
and  in particular, $\sum_n (T_n)_{\ell_p} =\sum_n t_n+k_n$ converges to $T_{|\ell_p}$. We deduce that $\sum_n k_n$ norm
converges to $T_{|\ell_p}-t$. Therefore
$T_{|\ell_p}-t$ is compact, which means that $t$ belongs to $ext_K(\ell_p,Y)$.

\end{proof}

\begin{remark}
It has been proved in \cite{CCFM} that whenever $Y$ is a non trivial twisted Hilbert space, there exist complex structures on $\ell_2$ which are not extendable to an operator on $Y$; therefore $ext(\ell_2,Y) \neq L(\ell_2,Y)$, which implies that
$ext(\ell_2,Y)$ is not closed.
\end{remark} 

We now improve this result to showing that
$ext_K(\ell_2,Y)$ is not closed, whenever $Y$ is a twisted Hilbert spaces which is not ``asymptotically" trivial, including Kalton-Peck space $Z_2$.

Recall that a quasilinear map $\Omega$ is $C$-trivial when $\Omega-L$ has norm at most $C$ for some linear $L$.

A quasi-linear map $\Omega$ on a Banach space $X$ with a symmetric (or subsymmetric) basis is called \emph{symmetric} (respectively, \emph{subsymmetric}) if there exists a constant $C>0$ such that 
\[
\|\, \Omega(x \circ \sigma) - \sigma \circ \Omega(x)\,\| \leq C \|x\|,
\]
for all $x \in X$ and every permutation (respectively, increasing injection) $\sigma$ of $\mathbb{N}$, where $\sigma$ acts on $X$ by permuting the coordinates according to the basis \cite{kaltondifferentials}. This means that $\Omega$ behaves in an approximately equivariant way with respect to the natural symmetry (or subsymmetry) of the underlying basis.

\begin{definition} A quasi-linear map $\Omega$ on $X$ is asymptotically trivial if there exists $C$ such that for any $n$ there exists a finite codimensional subspace $Y$ of $X$ such that for any $n$-dimensional subspace $F$ of $Y$, $\Omega_{|F}$ is $C$-trivial.
In this case we shall also say that the induced twisted sum is asymptotically trivial.
\end{definition}

The main examples of non-asymptotically trivial sums are (a) self-extensions of $\ell_p$ for which $\Omega$ is non-trivial and (sub)symmetric, including Kalton-Peck space $Z_p$, or more generally (b) any non trivial twisted Hilbert space induced by interpolation of a space with subsymmetric basis with its dual. Other examples are (c) twisted Hilbert spaces induced by interpolation of a space with asymptotically $\ell_p$-basis with its dual, $p \neq 2$. On the other hand, Suarez's weak Hilbert twisted Hilbert space \cite{Suarez} is asymptotically trivial but non-trivial. Also, it will be useful to recall from \cite{CCFM} that a quasilinear map $\Omega$ is {\em supersingular} if for any $C>0$ there exist $n \in \N$ such that for any subspace $F$ of dimension $n$, $\Omega_{|F}$ is not $C$-trivial; so supersingular maps cannot be asymptotically trivial.

\begin{lemma}\label{extension}  Let $F_k, k \in \N$ 
be a partition of $\N$ into successive subsets of cardinality $n$. Let $X, Y$ be  spaces with a (sub)symmetric basis and $t$ be an operator on $X$ such that the subspaces
$[e_i, i \in F_k]$ are $t$-invariant for all $k \in \N$.
If $\Omega: X \rightarrow Y$ is a (sub)symmetric quasilinear map,
then $t$ is extendable and liftable to $X \oplus_\Omega Y$. 
\end{lemma}

\begin{proof} For  $A=\{a_k, k \in \N\}$,
$B=\{b_k, k \in \N\}$, with $a_k, b_k \in F_k$, let $t_{A,B} \in L(X,Y)$ be defined by $t_{A,B}(e_{a_k})=e_{b_k}$ for all $k \in \N$, and $t_{A,B}(e_i)=0$ if $i \notin A$.
 The operator $t$ can be written as a linear combination of $n^2$ operators of the form $t_{A,B}$, and each operator $t_{A,B}$ is extendable and liftable to
 $T_{A,B}$ defined by $T_{A,B}(x,y)=(t_{A,B}x.t_{A,B}y)$
\end{proof}

Next we need a quantified version of \cite{CCFM} Lemma 3.2, to characterize $C$-triviality in terms of certain averages $\nabla \Omega$. If $\Omega$ is quasilinear and $(u_i)_i$ is a finite sequence of vectors, the authors of \cite{CCFM} define 
$$\nabla_{(u_i)_i}\Omega={\rm Ave}_{\varepsilon_i=\pm 1}
\|\Omega(\sum_i \varepsilon_i u_i)-\sum_i \varepsilon_i \Omega(u_i)\|.$$

\begin{lemma}\label{trivialNew}Let $H$ be  a  Hilbert space, $\Omega: H \rightarrow H$  be a quasi-linear map, and $W$ a closed subspace of $H$. There are universal constants $a,A,B$ such that
\begin{enumerate}
\item[(i)]
If $\Omega_{|W}$ is $C$-trivial, then for every  finite sequence $x=(x_k)^n_{k=1}$  of  normalized vectors in $W$ and every finite sequence of scalars $\lambda =(\lambda _k)_{k=1}^n $,
$$   \nabla_{[\lambda x]} \Omega\leq aC  \|  \lambda \|_2.
$$
\item[(ii)]
conversely, if    $\nabla_{[\lambda x]} \Omega\leq C  \|  \lambda \|_2$ for all $x$ in $W$ and $\lambda$ as above, then $\Omega_{|W}$ is $AC+B$-trivial.
\end{enumerate}
\end{lemma}

\begin{proof} Suppose  that  $\Omega|_W$ is $C$-trivial for a closed subspace $W$ of $H$. Then we  can write $\Omega|_W=B+L$  with $B$ bounded homogeneous and $L$ linear. Let $x=(x_k)_{k=1}^n$ be a  normalized sequence of  vectors in $W$ and $\lambda =(\lambda _k)_{k=1}^n $ be a sequence of scalars. Since $H$ has type 2 one gets:
\begin{eqnarray*}
 \nabla_{[\lambda x]} \Omega   &=& {\rm Ave}_{\pm} \left \|  B \left ( \sum_{k=1}^{n} \varepsilon_k \lambda_k  x_k\right ) - \sum_{k=1}^{n} \varepsilon_k B(\lambda_k  x_k)\right \| \\
&\leq& \|B\|{\rm Ave}_{\pm} \left \| \sum_{k=1}^{n}  \varepsilon_k  \lambda_k x_k \right \| +Ave_{\pm}  \left \|  \sum_{k=1}^{n}  \varepsilon_kB  \lambda_k x_k \right \| \\
&\leq& \|B\| k_2\left(  \sum_{k=1}^{n}   \| \lambda_k  x_k\| ^2\right )^{1/2} +  k_2\left(  \sum_{k=1}^{n} \|B \lambda_k x_k\| ^2\right )^{1/2}\\
&\leq& 2k_2C  \|\lambda\|_2.
\end{eqnarray*}
Here $k_2$ is (one of )the type $2$ constant(s) of $\ell_2$.
Conversely,  let $(y_i, \lambda_iw_i)_{i=1}^n$ be a finite sequence of vectors in $H\oplus_{\Omega} W$, with $w=(w_i)_{i=1}^n$ normalized,   and $\varepsilon=(\varepsilon_i)_{i=1}^n$  be a sequence of signs, then 

\begin{eqnarray*}
\left \| \sum_{i=1}^n \varepsilon_i (y_i, \lambda_iw_i) \right \|_{\Omega}  &=&  \left \|  \sum_{i=1}^n \varepsilon_i y_i - \Omega \left (  \sum_{i=1}^n \varepsilon_i \lambda_iw_i \right ) \right \|  + \left \|   \sum_{i=1}^n \varepsilon_i \lambda_iw_i\right \|\\
&\leq&  \left  \|  \sum_{i=1}^n \varepsilon_i \Omega (\lambda_iw_i)- \Omega \left (  \sum_{i=1}^n \varepsilon_i \lambda_iw_i \right ) \right \| + \left \|  \sum_{i=1}^n \varepsilon_i \left(y_i - \Omega(\lambda_iw_i) \right)\right \| +  \left \|   \sum_{i=1}^n \varepsilon_i\lambda_iw_i\right \|. 
\end{eqnarray*}
Now, by taking the average in both sides and using that $H$ has type 2, we have

\begin{eqnarray*} 
{\rm Ave}_{\pm}  \left \| \sum_{i=1}^n \varepsilon_i (y_i, \lambda_i w_i)  \right \|_{\Omega}   &\leq &   \nabla_{ [ \lambda w]} \Omega  +  k\left (  \sum_{i=1}^n \| y_i- \Omega (\lambda_iw_i)\|^2 \right )^{1/2} + k\left (  \sum_{i=1}^n \|\lambda_iw_i\|^2 \right )^{1/2}\\
&\leq&  (C+4k_2) \left (  \sum_{i=1}^n \left \| (y_i, \lambda_iw_i) \right \|_{\Omega} ^2\right )^{1/2}.
\end{eqnarray*}
Therefore  $H\oplus_{\Omega} W$ has type 2 with constant $C+4k_2$. It follows from the Maurey's extension theorem, \cite{Ma} Corollaire 4, that H is $c(C+4k_2)$-complemented in $H\oplus_{\Omega} W$, for some universal constant $c$, and then  the restriction of $\Omega$ to $W$ is $c(C+4k_2)$-trivial. 
\end{proof}

Recall that a pair of maps $(\tau,T)$ on spaces $X$ and $Y$, respectively, is said to be {\em compatible} (with a quasilinear map $\Omega$) if there is a lifting $\widetilde{T}$ of $T$ to $X \oplus_\Omega Y$ with induced map $\tau$ on $X$ by restriction. It is {\em $C$-compatible} when $\widetilde{T}$ may be chosen with norm at most $C$.
 Classical compatibility estimates (see for example \cite{CCFM}) imply that if $(\tau,T)$ is $C$-compatible, then
$$\Omega T-\tau \Omega$$ is $C$-trivial. This follows from seeing $\widetilde{T}$ as a matrix $\begin{pmatrix} \tau & L \\ 0 & T\end{pmatrix}$ and writing the estimate
$$\|\widetilde{T}\|  \geq \|\widetilde{T}(\Omega x, x)\|_\Omega=\|(\tau \Omega x + Lx, Tx)\|_\Omega=\|Tx\|+\|(\tau \Omega - \Omega T +L)x\|,$$ for all normalized $x$.

\begin{proposition}\label{main} Assume $Y=\ell_2 \oplus_\Omega \ell_2$ is a non asymptotically trivial twisted Hilbert space, and let $n \in \N$. Then there is a map $t_n$ on $\ell_2$ such that any lifting of a compact perturbation of $t_n$ to $Y$ has norm at least $n$. If $\Omega$ is (sub)symmetric then $t_n$ may be chosen to be liftable.
\end{proposition}

\begin{proof} 
We follow the ideas of  Theorem 3.1 in \cite{CCFM}.
We claim there is an orthogonal decomposition $\ell_2=H_1 \oplus H_2$ (using subspaces generated by subsequences of the basis) such that
$\Omega_{|H_1}$ is not asymptotically trivial and $\Omega_{|H_2}$ is not supersingular.
Indeed if for example $\Omega_{|H_1}$ is asymptotically trivial then it is not supersingular and $\Omega_{|H_2}$ cannot be asymptotically trivial, and we are done.
If neither $\Omega_{|H_1}$ nor $\Omega_{|H_2}$ are asymptotically trivial, then since as in the proof of \cite{CCFM} Theorem 3.1 one of the two must be non supersingular, we are also done.

Since $\Omega_{|H_1}$ is not asymptotically trivial, for a constant $c_n$ to be specified later, there exists $n_0$ such that every finite codimensional subspace of $H_1$ contains a subspace $E$ of dimension $n_0$ such that the restriction $\Omega_{|E}$ is not $c_n$-trivial. 
We consider  finite disjoint sets $F_k, k \in \N$ of cardinality $n_0$. 
Therefore by Lemma \ref{trivialNew} we can find normalized
$(e_i)_i$ in $H_1$ such that $[e_i, i \in F_k], k \in \N$ are successive and for each $k$, $$\nabla_{[(\lambda_i e_i)_{i \in F_k}]}\Omega \geq c_n $$ for some normalized $(\lambda_i)_{i \in F_k}$.
On the other hand by the definition of supersingularity \cite{CCFM}, we can pick normalized $(f_i)_i$ in $H_2$ and $C>0$ such that  $\Omega_{|[f_i, i \in F_k]}$ is $C$-trivial for each $k$. We may also assume that $[f_i, i \in F_k], k \in \N$ are successive. Therefore by Lemma \ref{trivialNew} and up to a change of constant, $\nabla_{[(\lambda_i f_i)_{i \in F_k}]}\Omega \leq C$ whenever $\| (\lambda_i)_i\|_2=1$. We may also assume the sequences $(e_i)_i$ and $(f_i)_i$ are  mutually disjointly supported in $\ell_2$.

Note that whenever $U$ is an operator on $\ell_2$ admitting a lifting $\widetilde{U}$ to $Y$  with induced map $u$ on $\ell_2$ by restriction, then applying Lemma \ref{trivialNew} to the $\|\widetilde{U}\|$-trivial map $\Omega U- u\Omega$, we get
$$\nabla_{[\lambda_i f_i, i \in F_k]}(\Omega U - u\Omega) \leq 2k_2\|\widetilde{U}\| \|(\lambda_i)_{i \in F_k}\|_2.$$
By a triangle inequality we obtain for all normalized $(\lambda_i)_{i \in F_k}$
$$\nabla_{[\lambda_iUf_i, i \in F_k]}\Omega \leq C\|u\|  + 2k_2\|\widetilde{U}\| \leq (C+2k_2)\|\widetilde{U}\|.$$
We denote by $M_n$  the supremum value of 
$\nabla_{[\lambda_ 1 x_1,\ldots,\lambda_n x_n]}\Omega$, where $x_1,\ldots,x_n$ range over all $x_i$'s of norm at most $1$ and $(\lambda_1,\ldots,\lambda_n)$ over all elements of $\|.\|_2$-norm $1$. 
Note that $M_n<\infty$ because $\Omega$ is quasi-linear. 
Pick now $T$ to be any operator on $\ell_2$ with $Tf_i=e_i$ for all $i$, and assume that $K$ is a compact operator on $\ell_2$ such that $T+K$ admits a lifting $\widetilde{T}$ with induced map $\tau$. Then in the above formula, since $(T+K)f_i=e_i+Kf_i$ we obtain
for all normalized $(\lambda_i)_{i \in F_k}$, $$\nabla_{[\lambda_i e_i+\lambda_i Kf_i, i \in F_k]}\Omega \leq (C+2k_2)\|\widetilde{T}\|.$$

Then by the triangular inequality and the fact the $\Omega$ is quasilinear:
$$\nabla_{[\lambda_i e_i, i \in F_k]}\Omega \leq \nabla_{[\lambda_i Kf_i, i \in F_k]}\Omega+ (C+2k_2)\|\widetilde{T}\|  + 2k_2 c(\Omega)
(1+\|K_{|[f_i, i \in F_k]}\|) 
 $$
 
Given $\varepsilon>0$ and taking $k$ large enough, we have
$$\|K_{|[f_i, i \in F_k]}\| \leq \varepsilon,$$ therefore, using the homogeneity of $\Omega$ and the fact that each $F_k$ has cardinality $n_0$,
we have  $$\nabla_{[\lambda_i Kf_i, i \in F_k]}\Omega \leq \varepsilon M_{n_0}.$$
It follows that
$$c_n \leq \varepsilon M_{n_0} + (1+\varepsilon)2k_2 c(\Omega) +(C+2k_2)\|\tilde{T}\|,$$
and therefore
$$c_n \leq 2k_2 c(\Omega)+(C+2k_2)\|\widetilde{T}\| .$$
Therefore no compact perturbation of $T$ can be lifted with to an operator $\widetilde{T}$ norm less than ${C+2k_2}^{-1} (c_n-2k_2 c(\Omega)$, and we only need to choose $c_n$ so that the previous expression is greater than $n$.

In the case of a (sub)symmetric $\Omega$, after the choice of $(e_i)_{i \in F_1}$ and $(f_i)_{i \in F_1}$, we may by subsymmetry assume that $(e_i)_{i \in F_k}$ and $(f_i)_{i \in F_k}$ are obtained by a map induced by permutation of the basis of $\ell_2$ from $(e_i)_{i \in F_1}$ and $(f_i)_{i \in F_1}$, and one can extend $T$ trivially to the whole of $\ell_2$. This implies that Lemma \ref{extension}
applies, and so $T$ may be assumed to be liftable (and extendable) to $Y$.
\end{proof}

Note that one can easily  make the construction so that $t_n$ is, for example, a complex structure, as in \cite{CCFM}, a symmetry, or other desired properties.

\begin{corollary}\label{exx} Assume $Y=\ell_2 \oplus_\Omega \ell_2$ is a twisted Hilbert space, with $Y^*$ non asymptotically trivial, and let $n \in \N$. Then there is a map $t_n$ on $\ell_2$ such that any extension of a compact perturbation of $t_n$ to $Y^*$ has norm at least $n$. If $\Omega$ was (sub)symmetric then $t_n$ may be chosen to be extendable.
\end{corollary}

\begin{proof} A simple duality argument. \end{proof}

\begin{corollary}\label{examp} Assume $Y=\ell_2 \oplus_\Omega \ell_2$ is a twisted Hilbert space, with $Y^*$  non asymptotically trivial. Then $ext_K(\ell_2,Y) \neq L(\ell_2,Y)$. Actually there is a map  $T: \ell_2 \rightarrow \ell_2 \subseteq Y$ which does not belong to $ext_K(\ell_2,Y)$.
\end{corollary}

\begin{proof}  
We consider a partition of $\N$ in finite sets $F_n$ such that each possible even cardinality is repeated infinitely many times, and sequences $(e_n)_n$ and $(f_n)_n$ which are orthonormal and mutually disjointly supported in $\ell_2$,
such that
$(e_i)_{i \in \cup_{|F_k|=n} F_k }$ and
$(f_i)_{i \in  \cup_{|F_k|=n} F_k }$ of the type considered in Proposition \ref{main}. We then define $T$ such that $T(f_i)=e_i$ for all $i \in \cup_k F_k$.
It follows that no compact perturbation of $T$ admits a lifting to $Y^*$. The result then follows by duality. 
\end{proof}

\begin{corollary}\label{inequiv} For every $n$, there exists an operator $t_n$ on $\ell_2$ such that $t_n \in ext(\ell_2,Z_2)$, $\|t_n\| \leq 1$, but any extension of a compact perturbation of $t_n$ to $Z_2$ has norm at least $n$.
\end{corollary}

\begin{proof} Indeed the Kalton-Peck map is symmetric. Actually 
in this case, one can refer to \cite{CCFM} for an explicit choice of $(e_n)_n$ and $(f_n)_n$: $(e_n)_n$ may be chosen to be a subsequence of the canonical basis of $\ell_2$, implying $\nabla_{[e_i, i \in F_k]}\Omega=\frac{1}{2}\sqrt{n}\log n$ and therefore $c(n)=\frac{1}{2}\log n$; the sequence $(f_n)_n$ may be defined as Rademacher-like vectors on the basis (see \cite{CCFM} Section 3.2 for the formula)  satisfying
$\nabla_{[f_i, i \in F_k]}\Omega \leq K_2 \sqrt{n}$ as a consequence of Khintchine's inequality.
\end{proof}

As observed earlier, $t_n$ may be chosen to be a complex structure on $\ell_2$
 in Corollary \ref{inequiv}.

\begin{proposition}\label{notclosed} Let $Y=\ell_2 \oplus_\Omega \ell_2$ be a twisted Hilbert space in which $\Omega$ is subsymmetric and non-trivial. There is an operator $T$ on $\ell_2$ which is not in $ext_K(\ell_2,Y)$ but belongs to the $\|.\|$-closure of $ext(\ell_2,Y)$. 
\end{proposition}

\begin{proof}
We use $t_n$ from Corollary \ref{exx}, by defining a partition $\N=\cup_n I_n$, increasing bijections $\sigma_n: \N \rightarrow I_n$, the canonical associated isometric maps $S_n: \ell_2 \rightarrow [e_i, i \in I_n]$ and projections $P_n: \ell_2 \rightarrow [e_i, i \in I_n]$.
Then we let $T=\sum_n  n^{-1/2} S_nt_nS_n^{-1}P_n$. The partial sums defining $T$ are extendable, so
$T$ belongs to $clos(ext(\ell_2,Y))$, and 
if some operator $\widetilde{T}$ extends a compact perturbation of $T$ then $n^{1/2}S_n^{-1}\widetilde{T}S_n$ extends a compact perturbation of $t_n$, implying
that $\|\widetilde{T}\| \geq \sqrt{n}$ for all $n$; so no compact perturbation of $T$ is extendable.
\end{proof}

We are now able to answer negatively the first question of Kalton and Swanson from \cite{kaltonasymplectic}.

\begin{corollary}\label{nor}
$(L(Z_2)/S(Z_2),\|.\|_S)$ is not complete. Equivalently the image of $L(Z_2)$ by $\Lambda$
 is not closed in $L(H)$.  \end{corollary}

\begin{proof} This is a consequence of Proposition \ref{equivalence} and Proposition \ref{notclosed}.
An equivalent and more illustrative argument is as follows, along the lines of Proposition \ref{principle}, which $Z_2$ satisfies: 
if $T_n$ is an extension of $t_n$ from Corollary \ref{inequiv}, then
$\|T_n\|_{L(Z_2)/S(Z_2)} \geq n$ (because every strictly singular perturbation of $T_n$ restricts to a compact perturbation of $t_n$). On the other hand, $\|T_n\|_S=\|t_n\|_{\mathrm{Calkin}} \leq 1$. Therefore $\|.\|_S$ is not equivalent to the quotient norm on $L(Z_2)/S(Z_2)$.
\end{proof}

It therefore follows that $Z_2$ provides a new example of algebra with inequivalent algebra norms. They are actually *-algebra norms, since the involution * is an isometry with respect to both of these norms.

\begin{theorem}\label{630} The *-algebra $L(Z_2)/S(Z_2)$ admits two inequivalent *-algebra norms. 
\end{theorem}

\begin{proof} The $\|.\|_S$ induces an *-algebra norm by Corollary \ref{KS}. That the usual quotient norm also induces a *-algebra norm follows from the fact that for an operator $S \in L(Z_2)$, $S$ is strictly singular if and only if its symplectic adjoint $S^{*}$ is strictly singular (\cite[Lemma 9]{kaltonasymplectic}) and the fact that $\|S^*\|=\|S\|$.
\end{proof}

That these are two inequivalent algebra norms is more generally true of all subsymmetric singular twisted Hilbert spaces; however we do not know whether they are *-algebra norms in this more general setting. 

\

Note that because of the *-isometric embedding $\widetilde{\beta}$, the normed algebra $(L(Z_2)/S(Z_2),\|.\|_S)$ has a completion which is a $C^*$-algebra (in particular it satisfies the relation $\|a^*a\|=\|a\|^2$). According to classical terminology,\cite{Rickart} Chapter 4 p181, this means that $(L(Z_2)/S(Z_2),\|.\|)$ is an ``$A^*$-algebra", for which $\|.\|_S$ is an ``auxiliary norm". This implies in particular that the quotient norm is the unique Banach algebra norm up to equivalence \cite{Rickart} Corollary 4.1.18.

In this line our techniques also give an answer to a second question of Kalton-Swanson, which also appears as Castillo-González-Pino \cite{raul}. They asked whether $L(Z_2)/S(Z_2)$  is *-isomorphic to a $B^*$-algebra (resp. $C^*$-algebra),  equipped with the usual quotient norm $\|.\|$. Note that this would follow from $L(Z_2)/S(Z_2),\|.\|_S$ being complete (since then $\|.\|$ and $\|.\|_S$ would be equivalent), but we know this is not the case.

\begin{theorem}\label{631} The *-algebra $L(Z_2)/S(Z_2)$, equipped with the quotient norm, is not *-isomorphic to a $B^*$-algebra.
\end{theorem}

\begin{proof} Let $t$ be an isometric complex structure on $\ell_2$ which extends to a map $T$ on $Z_2$ and such that, if $a=\tilde{T} \in L(Z_2)/S(Z_2)$, $\|a\| \geq n \in \N$ (Corollary \ref{inequiv}, the comment after it, and the proof of Corollary \ref{nor}).
Then $\beta(T)$ is a unitary operator on $W$, therefore
$\beta(T^*T)=\beta(T)^*\beta(T)=Id_W$. Therefore
$\beta(T^*T-Id)=0$ and therefore we have $a^*a=1$, and $\|a^*a\|=1$. On the other hand $\|a\|^2 \geq n^2$.
So $\|.\|$ cannot be $n$-equivalent to a $B^*$-algebra norm (and therefore neither to a $C^*$-algebra norm).
\end{proof}

\section{Questions and comments}

About images of complex structures by $\beta$, we have the following observation, following from  Proposition \ref{distance}:
\begin{remark} Let $Y$ be a singular twisted Hilbert space, then 
 $I:=\beta(\tilde{\mathcal I}(Y))$ and $J:=\beta(\tilde{\mathcal J}(Y))$ are sets of complex structures on the Hilbert space $W$ with $d(I,J) \geq 2$. 
\end{remark}

Our proof that $\|.\|_S$ is a *-algebra norm on $L(Z_2)/S(Z_2)$ relies heavily on specific properties of $Z_2$, and so it is natural to ask:

\begin{question}
Is $\|.\|_S$  is a *-algebra norm on $L(Y)/S(Y)$ whenever 
$Y$ is a subsymmetric singular twisted Hilbert space?  
\end{question}

Since the twisted Hilbert space $X$ considered by J. Suarez \cite{Suarez}, induced by the interpolation of the 2-convexified Tsirelson space with its dual, is asymptotically trivial, our extension techniques do not seem to apply, and therefore we also ask:
\begin{question} Is $L(X)/S(X)$ complete for $\|.\|_S$? Or is the quotient norm on $L(X)/S(X)$ equivalent to a $C^*$-norm? 
\end{question}

Recall  \cite{CCFM} that no complex structure on $\ell_2$ can extend to a complex structure on a hyperplane of $Z_2$. More generally:

\begin{question} If $j$ is a complex structure on $\ell_2$ and $J$ a complex structure on a hyperplane of $Z_2$ containing $\ell_2$, does it follow that $J_{|\ell_2}-j$ is not strictly singular? Or even that $\|J_{|\ell_2}-j\|_S \geq c >0$ for some $c$ depending on $\|j\|$ and $\|J\|$?
\end{question}

A more thorough study of the seminorm $\|.\|_S$ remains to be done:

\begin{question}
Characterize Banach spaces $X$ for which the seminorm $\|.\|_S$ is equivalent, or even isometrically equivalent, to the quotient norm on $L(X)/S(X)$.    
\end{question}

\section{Acknowledgments}
The authors thank Niels Laustsen for comments and references concerning the problem of uniqueness of algebra norms.

 \par 
  \bigskip
  \textsc{\footnotesize W. Cuellar Carrera, 
Departamento de Matem\'atica, Instituto de Matem\'atica e
Estat\'\i stica, Universidade de S\~ao Paulo, rua do Mat\~ao 1010,
05508-090 S\~ao Paulo SP, BRAZIL}

\textit{E-mail address}: \texttt{cuellar@ime.usp.br}

  \par 
  \bigskip
  \textsc{\footnotesize V. Ferenczi, 
Departamento de Matem\'atica, Instituto de Matem\'atica e
Estat\'\i stica, Universidade de S\~ao Paulo, rua do Mat\~ao 1010,
05508-090 S\~ao Paulo SP, BRAZIL,
and \\
Equipe d'Analyse Fonctionnelle,
Institut de Math\'ematiques de Jussieu,
Sorbonne Universit\'e - UPMC,
Case 247, 4 place Jussieu,
75252 Paris Cedex 05,
FRANCE}

\textit{E-mail address}: \texttt{ferenczi@ime.usp.br}

\end{document}